\documentclass[11pt]{amsart}

\usepackage[a4paper,hmargin=3.5cm,vmargin=4cm]{geometry}
\usepackage{amsfonts,amssymb,amscd,amstext}
\usepackage{graphicx}
\usepackage[dvips]{epsfig}
\usepackage{pstricks,pst-plot}

\usepackage{fancyhdr}
\input xy
\xyoption{all}

\usepackage{times}

\usepackage{enumerate}
\usepackage{titlesec}
\usepackage{mathrsfs}

\titleformat{\section}
{\filcenter\bfseries\large} {\thesection{.}}{0.2cm}{}
\titleformat{\subsection}[runin]
{\bfseries} {\thesubsection{.}}{0.15cm}{}[.]
\titleformat{\subsubsection}[runin]
{\em}{\thesubsubsection{.}}{0.15cm}{}[.]

\usepackage[up,bf]{caption}

\newtheorem{theorem}{Theorem}[section]
\newtheorem{proposition}[theorem]{Proposition}

\newtheorem{lemma}[theorem]{Lemma}
\newtheorem{corollary}[theorem]{Corollary}

\theoremstyle{definition}
\newtheorem{definition}[theorem]{Definition}
\newtheorem{remark}[theorem]{Remark}

\numberwithin{equation}{section}
\numberwithin{figure}{section}

\newcommand\Bcal{\mathcal{B}}

\newcommand\Ocal{\mathcal{O}}
\newcommand\Pcal{\mathcal{P}}

\newcommand\Ascr{\mathscr{A}}

\newcommand\Cscr{\mathscr{C}}

\newcommand\Oscr{\mathscr{O}}

\newcommand\Gscr{\mathscr{G}}

\newcommand\Pscr{\mathscr{P}}

\newcommand\C{\mathbb{C}}

\newcommand\CP{\mathbb{CP}}

\newcommand\N{\mathbb{N}}

\newcommand\R{\mathbb{R}}

\renewcommand\S{\mathbb{S}}

\newcommand\Z{\mathbb{Z}}

\renewcommand\c{\mathbb{C}}

\newcommand\cp{\mathbb{CP}}
\renewcommand\d{\mathbb D}

\newcommand\n{\mathbb{N}}
\renewcommand\r{\mathbb{R}}

\newcommand\z{\mathbb{Z}}

\newcommand\igot{\mathfrak{i}}

\renewcommand\igot{\mathfrak{i}}

\newcommand\pgot{\mathfrak{p}}
\newcommand\qgot{\mathfrak{q}}

\newcommand\Agot{\mathfrak{A}}

\newcommand\Mgot{\mathfrak{M}}

\renewcommand\imath{\igot}

\newcommand\Psf{\mathsf{P}}

\newcommand\hra{\hookrightarrow}
\newcommand\lra{\longrightarrow}

\newcommand\longhookrightarrow{\ensuremath{\lhook\joinrel\relbar\joinrel\rightarrow}}

\newcommand\wt{\widetilde}

\newcommand\di{\partial}
\newcommand\dibar{\overline\partial}

\renewcommand\span{\mathrm{span}}

\newcommand\Flux{\mathrm{Flux}}

\usepackage{color}

\begin{document}

\fancyhead[LO]{Every meromorphic function is the Gauss map of a conformal minimal surface}
\fancyhead[RE]{A.\ Alarc\'on, F.\ Forstneri\v c, and F.\ J.\ L\'opez}
\fancyhead[RO,LE]{\thepage}

\thispagestyle{empty}

\vspace*{7mm}
\begin{center}
{\bf \LARGE Every meromorphic function is the Gauss map of a conformal minimal surface}
\vspace*{5mm}

{\large\bf A.\ Alarc\'on, F.\ Forstneri\v c, and  F.\ J.\  L\'opez}
\end{center}

\vspace*{7mm}

\begin{quote}
{\small
\noindent {\bf Abstract}\hspace*{0.1cm}
Let $M$ be an open Riemann surface. We prove that every meromorphic function
on $M$ is the complex Gauss map of a conformal minimal immersion $M\to\r^3$ which may furthermore 
be chosen as the real part of a holomorphic null curve $M\to\c^3$. Analogous results are proved 
for conformal minimal immersions $M\to\r^n$ for any $n>3$. We also show that every conformal minimal immersion $M\to\r^n$ is 
isotopic through conformal minimal immersions $M\to\r^n$ to a flat one, and we identify the path connected components of 
the space of all conformal minimal immersions $M\to\R^n$ for any $n\ge 3$.

\vspace*{0.1cm}
\noindent{\bf Keywords}\hspace*{0.1cm} Riemann surface, complex curve, minimal surface, Gauss map.

\vspace*{0.1cm}

\noindent{\bf MSC (2010):}\hspace*{0.1cm} 49Q05, 30F99}

\vspace*{0.1cm}
\noindent{\bf Date: \today}

\end{quote}


\section{Introduction}\label{sec:intro}
Let $M$ be an open Riemann surface. The exterior differential on $M$ splits as the sum $d=\di+\dibar$ 
of the $\C$-linear part $\di$ and the $\C$-antilinear part $\dibar$.
Given a conformal minimal immersion $X=(X_1,\ldots,X_n)\colon M\to\r^n$ 
$(n\ge 3)$, the $1$-form $\di X=(\di X_1,\ldots,\di X_n)$ with values in $\C^n$ is holomorphic, 
it does not vanish anywhere on $M$, and it satisfies the nullity condition 
$\sum_{j=1}^n (\di X_j)^2=0$; we refer to Osserman \cite{Osserman-book} for these classical facts. Therefore, $\di X$ determines the Kodaira type holomorphic map
\begin{equation}\label{eq:GX}
	G_X\colon M\to \CP^{n-1},\quad G_X(p) = [\di X_1(p) \colon \cdots \colon \di X_n(p)]
	\ \ \ (p\in M).
\end{equation}
The map $G_X$ is known as the {\em generalized Gauss map} of $X$ and is of great importance
in the theory of minimal surfaces; see \cite{Osserman-book} and the papers \cite{Fujimoto1983JMSJ,Fujimoto1990JDG,HoffmanOsserman1980MAMS,LopezPerez2003IUMJ,OssermanRu1997JDG,Ros2002DG,Ru1991JDG}, 
among many others. Note that $G_X$ assumes values in the complex hyperquadric
\begin{equation}\label{eq:nullquadric-projected}
     Q_{n-2} = \bigl\{[z_1:\ldots : z_n]\in\CP^{n-1} : z_1^2+ \cdots + z_n^2 = 0\bigr\}. 
\end{equation}     

In this paper we prove the following result. 

%
%
\begin{theorem}\label{th:Gauss-n}
Let $M$ be an open Riemann surface and $n\geq 3$ be an integer.
For every holomorphic map $\Gscr\colon M\to Q_{n-2}\subset \cp^{n-1}$ into the quadric 
\eqref{eq:nullquadric-projected} there exists a conformal minimal immersion $X\colon M\to\r^n$ with
the generalized Gauss map $G_X=\Gscr$ and with vanishing flux. If in addition the map $\Gscr$ is full, 
then $X$ can be chosen to have arbitrary flux and 
to be an embedding if $n\ge 5$ and an immersion with simple double points if $n=4$. 
\end{theorem}

Recall that a map $\Gscr\colon M\to \cp^{n-1}$ is said to be {\em full} if its image
is not contained in any proper projective subspace. The {\em flux} of a conformal minimal immersion 
$X\colon M\to\R^n$ is the group homomorphism $\Flux_X\colon H_1(M;\z)\to\r^n$ given by
\begin{equation}\label{eq:flux}
     \Flux_X(\gamma)=\int_\gamma \Im(\di X)=-\imath\int_\gamma \di X\quad 
     \text{for every closed curve }\gamma\subset M.
\end{equation}
Here, $\imath=\sqrt{-1}$ and $\Re$, $\Im$ denote the real and the imaginary part, respectively.

We shall write $\C^n_*= \c^n\setminus\{0\}$ and $\C_*=\C\setminus\{0\}$.
Denote by $\pi\colon \c^n_* \to\cp^{n-1}$ the canonical projection
$\pi(z_1,\ldots,z_n)=[z_1:\cdots:z_n]$. Then, $Q_{n-2}=\pi(\Agot_*)$ where 
\begin{equation}\label{eq:nullquadric}
	\Agot=\big\{ (z_1,\ldots, z_n)\in \c^n : \sum_{j=1}^n z_j^2=0\big\}
\end{equation}
is the {\em null quadric} and $\Agot_*=\Agot\setminus\{0\}$. Fixing a nowhere vanishing holomorphic $1$-form
$\theta$ on $M$ (such exists by the Oka-Grauert principle), the holomorphic map $\di X/\theta\colon M\to \C^n$ 
assumes values in $\Agot_*$ and we have that $G_X= \pi\circ (\di X/\theta)$.

Since an open Riemann surface $M$ is homotopy equivalent to a wedge of circles,
every holomorphic map $\Gscr \colon M\to \cp^{n-1}$ lifts to a holomorphic map $f=(f_1,\ldots,f_n) \colon M\to \c^n_*$ 
such that $\Gscr = \pi\circ f= [f_1\colon\cdots\colon f_n]$
(see Lemma \ref{lem:lifting}). 
If $\Gscr(M)\subset Q_{n-2}$ then $f(M)\subset \Agot_*$.
Clearly, $\Gscr$ is full if and only if $\span(f(M))=\C^n$. 
The main idea behind the proof of Theorem \ref{th:Gauss-n}
is to find a nowhere vanishing holomorphic function $h\colon M\to \C_*$ such that
the $1$-form $\Phi=hf\theta$ with values in $\C^n$ integrates to a conformal minimal immersion $X\colon M\to\R^n$ 
with $\di X=\Phi$; hence $G_X=\Gscr$. This is the case if and only if the real periods of $\Phi$ vanish:
\begin{equation}\label{eq:period00}
	\int_\gamma \Re(\Phi)=0\quad \text{for all closed curves $\gamma$ in $M$}.
\end{equation} 
If this holds then, fixing a base point $p_0\in M$, $X$ is obtained by the formula 
\begin{equation}\label{eq:X}
   	X(p)=2\int_{p_0}^p \Re(\Phi)\quad \text{for all } p\in M.
\end{equation}
If $M$ is simply connected then \eqref{eq:period00} is satisfied for any holomorphic 
function $h$ on $M$, and hence in this case the first part of Theorem \ref{th:Gauss-n} is obvious.
However, if $M$ is topologically nontrivial, then the task becomes a fairly involved one; a suitable  
{\em multiplier} $h$ is provided by Theorem \ref{th:main1} which is the main technical result of the paper. 

In the case $n=3$, the quadric $Q_1\subset \CP^2$ \eqref{eq:nullquadric-projected} is the image of a quadratically
embedded Riemann sphere $\CP^1\hra \CP^2$, and the {\em complex Gauss map} of a conformal minimal
immersion $X=(X_1,X_2,X_3)\colon M\to\R^3$ is defined to be the holomorphic map 
\begin{equation}\label{eq:C-Gauss}
	g_X = \frac{\di X_3}{\di X_1-\imath \, \di X_2}  =  \frac{\di X_2-\imath\, \di X_1}{\imath\, \di X_3} : M \lra \CP^1.
\end{equation}
The function $g_X$ equals the stereographic projection of the real Gauss map
$N=(N_1,N_2,N_3) \colon M\to \S^2 \subset \r^3$ to the Riemann sphere $\CP^1$; explicitly,
\[
	g_X = \frac{N_1+\imath N_2}{1-N_3} : M \lra \C\cup\{\infty\}=\CP^1.
\]
We can recover the differential $\di X=(\di X_1,\di X_2,\di X_3)$ from the pair $(g_X,\phi_3)$
with $\phi_3=\di X_3$ by the classical Weierstrass formula
\begin{equation}\label{eq:EWR}
	\di X= \Phi=(\phi_1,\phi_2,\phi_3) =
	\left( \frac{1}{2} \left(\frac{1}{g_X}-g_X\right),  \frac{\imath}{2} \left(\frac{1}{g_X}+g_X \right),1\right) \phi_3.
\end{equation}
(See \cite[Lemma 8.1, p.\ 63]{Osserman-book}.) 
Conversely, given a pair $(g,\phi_3)$ consisting of a holomorphic map $g\colon M\to \CP^1$
and a meromorphic $1$-form $\phi_3$ on $M$, the meromorphic $1$-form 
$\Phi=(\phi_1,\phi_2,\phi_3)$ defined by \eqref{eq:EWR} satisfies $\sum_{j=1}^3 \phi_j^2=0$;
it is the differential $\di X$ of a conformal minimal immersion \eqref{eq:X} 
if and only if it is holomorphic, nowhere vanishing, and its real periods vanish.
Hence, Theorem \ref{th:Gauss-n} has the following immediate corollary.

\begin{corollary}\label{cor:Gauss-3}
Let $M$ be an open Riemann surface. Every holomorphic map $g\colon M\to\CP^1$ is the complex
Gauss map \eqref{eq:C-Gauss} of a conformal minimal immersion $X\colon M\to \R^3$ with vanishing flux.
If $g$ is nonconstant, then we can find $X$ with arbitrary given flux. 
\end{corollary}

Note that a nonconstant map $g \colon M\to \CP^1 \cong Q_1\subset \CP^2$ is full as a map into $\CP^2$.

The complex Gauss map of a minimal surface in $\r^3$ provides crucial information about its geometry. Several important
 properties of the surface depend only on its Gauss map, in particular, the Gauss curvature and the Jacobi operator 
(see e.g.\ \cite{MeeksPerez2004SDG,MeeksPerez2012AMS,Osserman1980DG,Osserman-book}). 
Thus, Corollary \ref{cor:Gauss-3} 
has applications to the theory of {\em stable minimal surfaces}. 
Recall that an immersed open minimal surface $S\subset\r^3$ is {\em stable} if any relatively compact smoothly bounded 
domain $D\subset S$ has the minimal area (in the induced metric) among all small variations of $\overline D$ which keep 
the boundary $bD$ fixed; equivalently, if the index of any such $D$ is zero. Let $X\colon M\to\r^3$ be a conformal 
minimal immersion and let $g_X$ denote its complex Gauss map \eqref{eq:C-Gauss}. It is classical 
(see Barbosa and do Carmo \cite[Theorem 1.2]{BarbosaDoCarmo1976AJM}) that the minimal surface $X(M)$ is stable
if the spherical image $g_X(M)\subset \CP^1$ of $X(M)$ has area less than $2\pi$. 
This holds for example if $g_X(M)$ lies in the unit disk $\d\subset\c$. 
Recall that every open Riemann surface $M$ carries a holomorphic function $h$ with no critical points 
(see Gunning and Narasimhan  \cite{GunningNarasimhan1967MA}). For any null vector $\nu \in \Agot^2_*$, the map 
$M\ni x\mapsto \Re(h(x)\nu) \in \r^3$ is a flat conformal minimal immersion with constant Gauss map, 
hence stable. In view of Barbosa and do Carmo \cite{BarbosaDoCarmo1976AJM}, 
Corollary \ref{cor:Gauss-3} gives the following more general result in this direction.

%
%
\begin{corollary}\label{co:stable}
If $M$ is an open Riemann surface and $g\colon M\to\cp^1$ is a holomorphic map whose image $g(M)$ has area less than $2\pi$, 
then there is a stable conformal minimal immersion $M\to\r^3$ with the complex Gauss map $g$.
\end{corollary}

We also prove the following result concerning isotopies (i.e., smooth $1$-parameter families)
of conformal minimal immersions into $\R^3$.

%
%

\begin{theorem}\label{th:avoiding}
Given an open Riemann surface $M$ and  a conformal minimal immersion $X\colon M\to\r^3$,
there exists an isotopy $X_t\colon M\to \R^3$ $(t\in [0,1])$ of conformal minimal immersions such that
$X_0=X$ and the complex Gauss map $g$ of $X_1$ \eqref{eq:C-Gauss} is nonconstant and avoids any 
two given points of the Riemann sphere. There also exists an isotopy $X_t$ as above such that $X_0=X$ and $X_1$ is flat.
\end{theorem}

Theorem \ref{th:avoiding} shows in particular that every conformal minimal immersion $M\to\r^3$ 
of an open Riemann surface can be deformed to a stable one.

In Section \ref{sec:path} we prove a more precise result to the effect that for any $n\ge 3$
the path connected components of the space of all conformal minimal immersions $M\to\R^n$ are in bijective
correspondence with the path connected components of the space of all {\em nonflat} conformal minimal
immersions $M\to\R^n$; see Theorem \ref{th:pathconnected}. There is only one connected 
component when $n>3$, but the situation is more complicated in dimension $n=3$.

The results presented above are proved in Section \ref{sec:finalsec} by using complex analytic methods. 
We now explain the main underlying technical result.

Let $M$ be an open Riemann surface and let $n\in\n$. A holomorphic map 
$f\colon M\to\c^n$ is said to be {\em full} if the image $f(M)$ does not lie in any affine hyperplane of $\c^n$. 
A holomorphic $1$-form $\Phi=(\phi_1,\ldots,\phi_n)$ on $M$ with values in $\C^n$ is said to be {\em full} 
if the map $\Phi/\theta\colon M\to\c^n$ is full, where $\theta$ is a holomorphic $1$-form vanishing nowhere on $M$; 
clearly the definition is independent of the choice of $\theta$. 

The following is the main technical result of this paper; it is proved in Section \ref{sec:proof}.

%
%
%
%
\begin{theorem}\label{th:main1}
Let $M$ be an open Riemann surface and let $n\in\n$ be an integer. Let $\Phi_t=(\phi_{t,1},\ldots,\phi_{t,n})$, $t\in [0,1]$, 
be a continuous family of full holomorphic $1$-forms on $M$ with values in $\C^n$, 
and let $\qgot_t\colon H_1(M;\z)\to\c^n$, $t\in [0,1]$, be a continuous family of group homomorphisms. 
Then there exists a continuous family of holomorphic functions $h_t\colon M\to\c_*$, $t\in [0,1]$, such that
\begin{equation}\label{eq:periods}
	\int_\gamma h_t\, \Phi_t=\qgot_t(\gamma)\quad \text{for every closed curve $\gamma\subset M$ and $t\in [0,1]$}.
\end{equation}
Furthermore, if the condition \eqref{eq:periods} holds at $t=0$ with the constant function $h_0=1$, then the homotopy
$h_t\colon M\to\c_*$ can be chosen with $h_0=1$.
\end{theorem}

Flux-vanishing conformal minimal surfaces in $\r^3$ admit an elementary deformation through the family of associated 
surfaces which share the same complex Gauss map. On the other hand, conformal minimal surfaces with vertical flux admit 
the L\'opez-Ros deformation (see \cite{LopezRos1991JDG}) which homothetically deforms the complex Gauss map while
preserving the third component. Recently, the first two named authors
proved that every conformal minimal immersion $M\to\r^3$ is 
isotopic to the real part of a holomorphic null curve $M\to\c^3$ (see \cite[Theorem 1.1]{AlarconForstneric2016CR}). 
Theorem \ref{th:main1} allows one to lift isotopies of full holomorphic maps $\Gscr_t\colon M\to Q_{n-2}\subset\cp^{n-1}$ $(t\in[0,1])$ 
to isotopies of conformal minimal immersions $X_t\colon M\to\r^n$ with the generalized Gauss maps $\Gscr_t$ and prescribed flux maps 
$\pgot_t\colon H_1(M;\z)\to\r^n$. In particular, we obtain the following stronger form of 
\cite[Theorem 1.1]{AlarconForstneric2016CR} in which the generalized Gauss map is preserved.

\begin{corollary}\label{cor:Crelle-sameGaussmap}
Let $M$ be an open Riemann surface and let $n\ge 3$ be an integer. Every conformal minimal immersion $X\colon M\to\R^n$
is isotopic through conformal minimal immersions $X_t\colon M\to\R^n$ $(t\in [0,1])$ 
to the real part $X_1=\Re Z$ of a holomorphic null curve $Z\colon M\to\C^n$ such that
all maps $X_t$ in the family have the same generalized Gauss map $G_X\colon M\to\CP^{n-1}$. 
Furthermore, if the generalized Gauss map $G_X$ of $X$ is full, then there is an isotopy $X_t$ as above 
such that $X_0=X$ and $X_1$ has any given flux.
\end{corollary}

It has been proved very recently in \cite{ForstnericLarusson2016} that the inclusion of the space of real parts
of all nonflat holomorphic null curves $M\to\C^n$ into the space of all nonflat conformal minimal immersions 
$M\to\R^n$ satisfies the parametric h-principle with approximation; in particular, it is a weak homotopy equivalence, 
and a strong homotopy equivalence if the homology group $H_1(M;\Z)$ is finitely generated. 
(Both spaces carry the compact-open topology.) However, the constructions in the papers 
\cite{AlarconForstneric2016CR,ForstnericLarusson2016} 
do not preserve the Gauss map, so this particular aspect of Corollary \ref{cor:Crelle-sameGaussmap} is new. 

In the proof of Theorem \ref{th:main1} we exploit the fact that $\C_*=\c\setminus\{0\}$ and the punctured
null quadric $\Agot_*$ \eqref{eq:nullquadric} are {\em Oka manifolds}. 
(See the survey \cite{ForstnericLarusson2011NY} for an introduction to Oka theory and the
monograph \cite{Forstneric2017E} for a comprehensive treatment.)
Furthermore, in order to achieve the correct periods of the $1$-forms $h_t\Phi_t$ (see \eqref{eq:periods}), 
we apply a technique similar to Gromov's {\em convex integration lemma} (compare to 
\cite[Section 3]{ForstnericLarusson2016} and the references therein), but using the 
$\C$-linear span instead of the convex hull.


\subsection*{Organization of the paper} 
In Sections \ref{sec:paths} and \ref{sec:sprays} we prove the technical lemmas
which are used to obtain the suitable family of multipliers $h_t$ in Theorem \ref{th:main1}.
In Section \ref{sec:proof} we prove Theorem \ref{th:main1}, and in Section \ref{sec:finalsec}
we show how it implies Theorems \ref{th:Gauss-n}, \ref{th:avoiding} and Corollary \ref{cor:Crelle-sameGaussmap}. 
In Section \ref{sec:structure} we prove that, for a compact bordered Riemann surface $M$, the space of all conformal minimal 
immersions $M\to\R^n$ with prescribed generalized Gauss map and flux carries the structure of a real analytic Banach manifold
(see Theorem \ref{th:structure}). Finally, in Section \ref{sec:path} we identify the path components of the space of all 
conformal minimal immersions $M\to\r^n$ for any $n\ge 3$ (see Theorem \ref{th:pathconnected}).


\subsection*{Notation} 
We shall use the notation $\c_*=\c\setminus\{0\}$, $\c^n_*=\c^n\setminus\{0\}$ $(n\ge 2)$, $\c^0=\{0\}$,
$\imath=\sqrt{-1}$, $\z_+=\{0,1,2,...\}$, $\N=\{1,2,\ldots\}$, and $\r_+=[0,+\infty)$. 

If $M$ is an open Riemann surface and $A\subset M$ is a subset, we denote by $\Ocal(A)$ the space of functions $A\to\c$ 
which are holomorphic on an open neighborhood (depending of the function) of $A$ in $M$. Similarly, by a holomorphic 
$1$-form on $A$ we mean the restriction to $A$ of a holomorphic $1$-form on an unspecified open neighborhood of $A$ in $M$. 

If $A$ is a compact smoothly bounded domain in $M$ and $r\in\z_+$, we denote by $\Ascr^r(A)$ the space of $\Cscr^r$ 
functions $A\to\c$ which are holomorphic in the interior $\mathring A=A\setminus bA$. 
Similarly, we define the spaces $\Ocal(A,Z)$ and $\Ascr^r(A,Z)$ of maps $A\to Z$ to a complex manifold $Z$. 
For simplicity we write $\Ascr(A)=\Ascr^0(A)$ and $\Ascr(A,Z)=\Ascr^0(A,Z)$.


\section{Multiplier functions on families of paths} 
\label{sec:paths}

In this section we prove a couple of technical lemmas which allow us to construct 
families of multipliers $h_t$ \eqref{eq:periods} in Theorem \ref{th:main1}. 
We first explain how to construct such multipliers on the interval $I=[0,1]\subset \R$; in the following section 
we use these results in the geometric setting which arises in the proof of Theorem \ref{th:main1}.
The main result of this section is Lemma \ref{lem:periods} whose proof proceeds in two steps: first we construct 
multipliers which give approximately correct values, and then we use Lemma \ref{lem:spray-loops} to correct
the error. 

Recall that a path $f\colon I=[0,1]\to\c^n$ is said to be {\em full} if the $\c$-linear span of its image
equals $\c^n$. The path is {\em nowhere flat} if for any proper affine subspace $\Sigma\subset \c^n$
the set $\{s\in I\colon f(s)\in \Sigma\}$ is nowhere dense in $I$. Note that a real analytic path which is
full is also nowhere flat; the converse holds for any continuous path.

For $n\in\n$ consider the period map $\Pcal\colon \Cscr(I,\c^n)\to\c^n$ defined by
\[
        \Pcal(f)=\int_0^1 f(s)\, ds\in\c^n,\quad f\in \Cscr(I,\c^n).
\]
We begin with the following existence result for {\em period dominating multiplier functions} for families of paths
$[0,1]\to\C^n$.

\begin{lemma}\label{lem:spray-loops}
Let $I'$ be a nontrivial closed subinterval of $I=[0,1]$,
let $Q$ be a compact Hausdorff space, and let $n\in\n$. 
Given a continuous map $f\colon Q\times I\to\c^n$ such that $f(q,\cdot)$ is full 
on $I'$ for every $q\in Q$, there exist finitely many continuous functions $g_1,\ldots,g_N\colon I\to\c$ $(N\ge n)$, 
supported on $I'$, such that the function $h\colon \c^N\times I\to\c$ given by
\[
             h(\zeta,s):=\prod_{i=1}^N \big(1+\zeta_i g_i(s) \big),\quad \zeta=(\zeta_1,\ldots,\zeta_N)\in\c^N,\; s\in I
\]
is a {\em period dominating multiplier of $f$}, meaning that the map
\begin{equation}\label{eq:period-domination}
      \frac{\di}{\di \zeta} \Pcal (h(\zeta,\cdot)f(q,\cdot)) \big|_{\zeta=0} : T_0\c^N\cong\c^N\to\c^n
      \ \ \text{is surjective for every $q\in Q$.}
\end{equation}
\end{lemma}

\begin{remark}\label{rem:stability}
Note that \eqref{eq:period-domination} is an open condition which remains valid with the same function $h$ 
if we replace $f$ by any $f' \in \Cscr(Q\times I, \c^n)$ sufficiently close to $f$.
\end{remark}

\begin{proof}
Let $N\ge n$ be an integer and, for each $i\in\{1,\ldots,N\}$, let $g_i\colon I\to\c$ be a continuous function supported 
on $I'$; both the number $N$ and the functions $g_i$ will be specified later. 
Let $\zeta=(\zeta_1,\ldots,\zeta_N)$ be holomorphic coordinates on $\c^N$. Set
\begin{equation}\label{eq:h-spray}
       h(\zeta,s):=\prod_{i=1}^N \big(1+\zeta_i g_i(s) \big),\quad (\zeta,s)\in \c^N\times I,
\end{equation}
and observe that
\begin{equation}\label{eq:dih}
       \left. \frac{\di h(\zeta,s)}{\di \zeta_i} \right|_{\zeta=0}=  g_i(s),\quad s\in I,\ i\in\{1,\ldots,N\}.
\end{equation}
Let $\wt\Pcal\colon Q\times \c^N\to \c^n$ be the map given by 
\[
	\wt\Pcal(q,\zeta)=\Pcal(h(\zeta,\cdot)f(q,\cdot))=\int_0^1 h(\zeta,s) f(q,s) \, ds,\quad (q,\zeta)\in Q\times\c^N.
\]
By \eqref{eq:dih} we have
\begin{equation}\label{eq:diQ}
     \left. \frac{\di \wt\Pcal(q,\zeta)}{\di \zeta_i}\right|_{\zeta=0}  = \int_0^1 \left. \frac{\di h(\zeta,s)}{\di \zeta_i}  \right|_{\zeta=0} f(q,s)\, ds  
     =   \int_0^1  g_i(s) f(q,s)\, ds.
\end{equation}

We now explain how to choose the functions $g_1,\ldots, g_N$.
Since $f(q,\cdot)$ is full 
on $I'$ for every $q\in Q$, compactness of $Q$ and continuity of $f$ ensure that 
there are distinct points $s_1,\ldots,s_N\in \mathring{I'}$ 
for a big $N$ such that
\begin{equation}\label{eq:span-N}
       \span\{f(q,s_1),\ldots, f(q,s_N)\}=\c^n\quad \text{for all $q\in Q$}.
\end{equation}
Let $\epsilon>0$ be small enough such that the intervals $[s_i-\epsilon,s_i+\epsilon]$ $(i=1,\ldots,N)$ are pairwise disjoint 
and contained in $I'$; the precise value of $\epsilon$ will be specified later. Let $g_i\colon I\to\c$ be any 
continuous function supported on $(s_i-\epsilon,s_i+\epsilon)\subset I'$ and satisfying 
\begin{equation}\label{eq:intg1}
  	 \int_0^1 g_i(s)\, ds=\int_{s_i-\epsilon}^{s_i+\epsilon} g_i(s)\, ds=1.
\end{equation}
To conclude the proof, it remains to show that the derivative
\[
	\frac{\di}{\di \zeta} \Pcal (h(\zeta,\cdot)f(q,\cdot)) \big|_{\zeta=0} : T_0\c^N\cong\c^N\to\c^n
\] 
is surjective for every $q\in Q$. Since $\Pcal(h(\zeta,\cdot)f(q,\cdot))=\wt\Pcal(q,\zeta)$, it suffices to prove that
\[
        \frac{\di}{\di \zeta} \wt\Pcal (q,\zeta) \big|_{\zeta=0} \colon T_0\c^N\cong\c^N\to\c^n\ \ 
       \text{is surjective for every $q\in Q$.}
\]
Indeed, for small $\epsilon>0$ we have in view of \eqref{eq:diQ} and \eqref{eq:intg1} that
\[
    \left. \frac{\di \wt\Pcal(q,\zeta)}{\di \zeta_i}\right|_{\zeta=0}=  \int_0^1 g_i(s) f(q,s)\, ds \approx f(q,s_i)\ \ 
    \text{for all $q\in Q$ and $i\in\{1,\ldots,N\}$.}
\]
Therefore, if $\epsilon>0$ is chosen small enough, condition \eqref{eq:span-N} guarantees that 
\[
	\span \left\{ \frac{\di \wt\Pcal(q,\zeta)}{\di \zeta_1}\Big|_{\zeta=0},\ldots, 
	\frac{\di \wt\Pcal(q,\zeta)}{\di \zeta_N}\Big|_{\zeta=0} \right\} = \C^n\quad 
	\text{for all $q\in Q$}.
\]
This concludes the proof of Lemma \ref{lem:spray-loops}.
\end{proof}

%
%

We now show the existence of multiplier functions for families of paths which enable us to prescribe 
the periods. Recall that $I=[0,1]$.

\begin{lemma}\label{lem:periods}
Let $f\colon I^2=I\times I \to \c^n$ and $\alpha\colon I\to \c^n$ be continuous maps.
Assume that the path $f_t:= f(t,\cdotp) \colon I\to\c^n$ is nowhere flat for every $t\in I$. 
Then there exists a continuous function $h\colon I^2\to \c_*$ such that $h(t,s)=1$ for $t\in I$ and $s\in  \{0,1\}$ and
\begin{equation}\label{eq:exact}
	\int_0^1 h(t,s) f(t,s)\, ds = \alpha(t), \quad t\in[0,1].
\end{equation}
If in addition we have that $\int_0^1 f(0,s)\, ds= \alpha(0)$,  then $h$ can be chosen such that
$h(0,s)=1$ for $s\in [0,1]$.
\end{lemma}

\begin{proof}
It suffices to prove that for any $\epsilon>0$ there exists a function $h\colon I^2\to \c_*$ such that
\begin{equation}\label{eq:approximate}
	\left| \int_0^1 h(t,s) f(t,s)\, ds - \alpha(t)\right| <\epsilon, \quad t\in[0,1].
\end{equation}
The exact result \eqref{eq:exact} can then be obtained by writing the parameter interval for the $s$-variable as
a union $I=I_1\cup I_2$, where $I_1$ and $I_2$ are nontrivial subintervals with a common endpoint 
(for example, $I_1=[0,1/2]$ and $I_2=[1/2,1]$) and applying the approximate result \eqref{eq:approximate} with a 
sufficiently small $\epsilon>0$ on $I_1$ and a period dominating argument on  $I_2$ 
(see Lemma \ref{lem:spray-loops}) in order to correct the error. 

Since $f_t$ is nowhere flat and hence full for each fixed $t\in[0,1]$, there is a division $0=s_0<s_1<\cdots <s_N=1$ of $I$ such that 
\[
	\span\{f_t(s_1),\ldots, f_t(s_N)\}=\c^n. 
\]
The same condition then holds for each $t' \in I$ sufficiently close to $t$. 
By adding more division points and using compactness of $I$ we obtain a division satisfying the above condition for all $t\in I$. Set 
\[
	V_j(t) = \int_{s_{j-1}}^{s_j} f_t(s)\, ds,\quad  j=1,\ldots,N. 
\]
Note that $V_j(t)$ is close to $f_t(s_j)(s_j-s_{j-1})$ if the segments are short.
By passing to a finer division if necessary we may therefore assume that
\[
	\span\bigl\{V_1(t),\ldots, V_N(t)\bigr\} =\c^n,\quad t\in I. 
\]
For each $t\in I$ we let $\Sigma_t\subset \c^N$ denote the affine complex hyperplane defined by 
\[
	\Sigma_t= \biggl\{(g_1,\ldots,g_N) \in \c^N : \sum_{j=1}^N g_j V_j(t) = \alpha(t) \biggr\}.
\]
Clearly, there exists a continuous map $g=(g_1,\ldots, g_N)\colon I\to \c^N$ such that 
$g(t)\in \Sigma_t$ for every $t\in I$. (We may view $g$ as a section of the affine bundle
over $I$ whose fiber over the point $t$ equals $\Sigma_t$.) This can be written as follows:
\[
	\sum_{j=1}^N \int_{s_{j-1}}^{s_j} g_j(t) f_t(s)\, ds = \alpha(t),\quad t\in I.
\]
Note that $\sum_{j=1}^N V_j(t) = \int_0^1 f_t(s)\, ds$. Hence, if $\int_0^1 f(0,s)\, ds= \alpha(0)$ 
then $g$ can be chosen such that $g(0)=(1,\ldots,1)\in \c^N$.
We assume in the sequel that this holds since the proof is even simpler otherwise.

By a small perturbation we may assume that $g_j(t)\in \c_*$ for every $t\in I$ and $j=1,\ldots, N$.
(At this point we need that the parameter space $I$ is one-dimensional.) This changes the
exact condition in the above display to the approximate condition 
\begin{equation}\label{eq:epsilon2}
	\left| \, \sum_{j=1}^N \int_{s_{j-1}}^{s_j} g_j(t) f_t(s)\, ds - \alpha(t)\right| < \frac{\epsilon}{2}, \quad t\in I.
\end{equation}
We shall now view the vector $g(t)=(g_j(t))_j\in \C^N$ for every fixed $t\in I$ as a step function of the
variable $s\in I$ which equals the constant $g_j(t)$ on the $j$-segment $s\in [s_{j-1},s_{j}]$ for every $j=1,\ldots,N$. 
Next, we approximate this step function by a continuous function $h_t=h(t,\cdotp)\colon I\to\c_*$ which agrees with the step
function, except near the division points $s_0,s_1,\ldots,s_{N}$ where we modify it in order to make
it continuous and to assume the value $1$ at the endpoints $0,1$ of $I$. 
Replacing the step function in \eqref{eq:epsilon2} by this new function $h(t,s)$ will cause
an  error of size $<\epsilon/2$ provided  the modification is supported on sufficiently 
short segments around the division points. This will yield the estimate \eqref{eq:approximate}.

We now explain the details. Let $C>1$ be chosen such that 
\[
	\max_{(t,s)\in I^2} |f(t,s)|\le C,\quad \max_{t\in I,\, j=1,\ldots,N} |g_j(t)| \le C.
\]
Due to simple connectivity of $I$ we can find for every $j=1,\ldots,N$ a homotopy of maps
$g_{j,\tau} \colon I\to \c_*$ $(0\le \tau \le 1)$ such that the following conditions hold:
\begin{itemize}
\item $g_{j,0}(t)=1$ for all $t\in I$,
\vspace{1mm}
\item $g_{j,1}(t)=g_j(t)$  for all $t\in I$,
\vspace{1mm}
\item $g_{j,\tau}(0)=1$ for all $\tau\in [0,1]$ (the homotopy is fixed at $t=0$), and
\vspace{1mm} 
\item $|g_{j,\tau}(t)| \le C$ for all $t\in I$ and $\tau \in [0,1]$. 
\end{itemize}
This holds for example if $g_{j,\tau}(t)=g_j(\tau t)$. Pick a number $\eta>0$ such that 
\begin{equation}\label{eq:eta}
	4C(C+1) N \eta < \epsilon.
\end{equation}
For each $t\in I$ and $j=1,\ldots N$ we define the function $h(t,\cdotp)\colon [s_{j-1},s_j]  \to \c_*$ as follows:
\[
	h(t,s) = \begin{cases} 
			g_{j,(s-s_{j-1})/\eta}(t), & s\in [s_{j-1},s_{j-1}+\eta]; \\
			g_j(t),                           & s\in [s_{j-1}+\eta,s_j-\eta]; \\
			g_{j,(s_j-s)/\eta}(t),      & s\in [s_{j}-\eta,s_j].
                    \end{cases} 
\]
This means that $h(t,\cdotp)$ spends most of its time (the middle segment $[s_{j-1}+\eta,s_j-\eta]$) at the point $g_j(t)$, 
and it travels between the point $1\in \c_*$ (where it is at the endpoints $s=s_{j-1}$ and $s=s_j$) 
and the point $g_j(t)$ along the trace of the path $\tau\mapsto g_{j,\tau}(t)\in\c_*$.
This defines a continuous function $h\colon I^2\to \c_*$ satisfying 
\[
	|h(t,s)|\le C\quad \text{for all $(t,s)\in I^2$.}
\]
It follows easily from \eqref{eq:epsilon2}, \eqref{eq:eta}, the definition of $h$ and the last estimate that 
$h$ satisfies the condition \eqref{eq:approximate}. This completes the proof.
\end{proof}


\section{Period dominating families of multipliers on admissible sets}
\label{sec:sprays}

The main result of this section is Lemma \ref{lem:existence-sprays} which provides small deformations 
of families of multipliers that make small but arbitrary changes in their integrals. This replaces
Lemma \ref{lem:spray-loops} (which pertains to multipliers on the interval $[0,1]$) in the geometric setting 
that arises in proving the inductive step in Theorem \ref{th:main1}. 
The proof of Lemma \ref{lem:existence-sprays} uses the construction from 
Lemma \ref{lem:spray-loops} together with the Mergelyan approximation theorem on admissible
sets in a Riemann surface; see Definition \ref{def:admissible}.
In the proof of Theorem \ref{th:main1} we shall combine Lemmas \ref{lem:periods} and \ref{lem:existence-sprays}.

We begin with some preparations.

\begin{definition}\label{def:admissible}
A nonempty compact subset $S$ of an open Riemann surface $M$ is said {\em admissible} if it is Runge in $M$ 
and of the form $S=K\cup \Gamma$, where $K$
is the union of finitely many pairwise disjoint  smoothly bounded compact domains in $M$ 
and $\Gamma:=\overline{S\setminus K}$ is a finite union of pairwise disjoint smooth Jordan arcs
meeting $K$ only at their endpoints (or not at all) and such that their intersections with the boundary 
$bK$ of $K$ are transverse.
\end{definition}

Let $S=K\cup\Gamma$ be an admissible subset of an open Riemann surface $M$. 
Given an integer $r\in\z_+$ and a complex submanifold $Z$ of $\c^n$,  we denote by 
\begin{equation}\label{eq:Fscr}
	\Ascr(S,Z)
\end{equation}
the set of all continuous functions $S\to Z$ which are holomorphic on $\mathring S =\mathring K$. 

Given a basis $\Bcal=\{C_1,\ldots,C_l\}$ of the homology group $H_1(S;\z)$, a holomorphic $1$-form $\theta$ 
vanishing nowhere on $M$, and a function $f\in\Ascr(S,\c^n)$ for some $n\in \n$, we define the {\em period map associated 
to $(\Bcal,f,\theta)$} as the map 
\begin{equation}\label{eq:periodmap}
    \Pcal^f=(\Pcal^f_1,\ldots,\Pcal^f_l) : \Ascr(S)\to (\c^n)^l
\end{equation}
given by
\begin{equation}\label{eq:periodmapj}
    \Pcal^f_j(h)=\int_{C_j} hf\theta\in\c^n,\quad h\in \Ascr(S),\ j=1,\ldots,l.
\end{equation}
It is clear that $\Pcal^f_j(h)$ lies in $\span(f(S))$ and it only depends on the homology class of $C_j$ for $j=1,\ldots,l$.
If $S$ is connected but not simply connected, then it is easily seen that there is a collection $C_1,\ldots,C_l$ of smooth 
Jordan curves in $\mathring S\cup\Gamma$ forming a homology basis $\Bcal$ of $S$ such that the support 
$|\Bcal|=\bigcup_{j=1}^l C_j$ of $\Bcal$ is a Runge subset of $M$ and each curve $C_j$ contains a nontrivial arc 
$\wt C_j$ which is disjoint from $C_k$ for all $k\neq j$.

%
%

Given a compact Hausdorff space $Q$, we let $\Ascr(Q\times S,Z)$ denote the space of 
continuous maps $f\colon Q\times S\to Z$ such that $f(q,\cdot)\in\Ascr(S,Z)$ for all $q\in Q$.

\begin{lemma}\label{lem:existence-sprays}
Let $S=K\cup\Gamma$ be a connected admissible subset of an open Riemann surface $M$,
and denote by $l\in\z_+$ the dimension of the first homology group $H_1(S;\z)$. 
Also, let $\theta$ be a   nowhere vanishing holomorphic $1$-form on $M$, and let $Q$ be a compact Hausdorff space. 
Assume that $f\colon Q\times S\to \c^n$ is a map of class $\Ascr(Q\times S)$ such that $f(q,\cdot)$ is full on $K$ 
and nowhere flat on $\Gamma$ for all $q\in Q$. There exist finitely many holomorphic functions $g_1,\ldots,g_N\in\Oscr(M)$ 
$(N\ge nl)$ such that the function $\Xi\colon \c^N\times M\to\c$ given by
\[
      \Xi(\zeta,p)=\prod_{i=1}^N \big(1+\zeta_i g_i(p)\big),\quad \zeta=(\zeta_1,\ldots,\zeta_N)\in\c^N,\; p\in M
\]
is a {\em period dominating multiplier of $f$}, meaning that for every $q\in Q$ the map
\begin{equation}\label{eq:period2}
    \c^N\ni\zeta\longmapsto \Pcal^{f,q}(\Xi(\zeta,\cdot))\in (\c^n)^l
\end{equation}
has maximal rank equal to $ln$ at $\zeta=0$.
(Here, $\Pcal^{f,q}$ is the period map associated to a fixed basis $\Bcal$ of $H_1(S;\z)$, the map $f(q,\cdot)$, and the
$1$-form $\theta$; see \eqref{eq:periodmap}, \eqref{eq:periodmapj}.)
\end{lemma}

\begin{proof}
If $S$ is simply connected then $l=0$ and hence $(\c^n)^l=\{0\}$. In this case the integer $N=1$ and the 
function $g_1\equiv 0$ (hence $\Xi\equiv 1$) satisfy the conclusion of the lemma.

Assume now that $S$ is not simply connected and so $l>0$.
Choose a collection $C_1,\ldots,C_l$ of smooth Jordan curves in $\mathring S\cup\Gamma$ forming a Runge homology basis 
$\Bcal$ of $S$ such that each curve $C_j$ contains a nontrivial arc $\wt C_j$ which is disjoint from $C_k$ for all $k\neq j$. 
For each $j=1,\ldots,l$ we fix a parameterization  $\gamma_j \colon [0,1] \to C_j$ with $\wt C_j\subset \gamma_j((0,1))$.
The assumptions on $f$ imply that the map $f(q,\cdot)\circ\gamma_j\colon[0,1]\to\c^n$ is nowhere flat for every 
$q\in Q$ and $j\in\{1,\ldots,l\}$. Denote by $|\Bcal|=\bigcup_{j=1}^l C_j\subset \mathring S\cup\Gamma$ the support 
of $\Bcal$. For each $q\in Q$ we denote by $\Pcal^{f,q}=(\Pcal^{f,q}_1,\ldots,\Pcal^{f,q}_l)\colon \Cscr(M)\to(\c^n)^l$
the map whose $j$-th component is given by
\[
       \Pcal_j^{f,q}(g)=\int_{C_j} gf(q,\cdot)\theta =\int_0^1 g(\gamma_j(s)) 
       f(q,\gamma_j(s))\theta(\gamma_j(s),\dot\gamma_j(s))\, ds, \quad g\in \Cscr(C_j).
\]
For each $j\in\{1,\ldots,l\}$, Lemma \ref{lem:spray-loops} furnishes an integer $N_j\ge n$ and continuous functions 
$g_{j,k}\colon C_j\to \c$ $(k=1,\ldots,N_j)$ supported on $\wt C_j$ such that the function 
\[
       h_j(\zeta_j,p)=\prod_{k=1}^{N_j}\big(1+\zeta_{j,k}g_{j,k}(p)\big),\quad 
       \zeta_j=(\zeta_{j,1},\ldots,\zeta_{j,N_j})\in \c^{N_j},\; p\in C_j
\]
satisfies
\begin{equation}\label{eq:diPjt}
       \frac{\di} {\di \zeta_j}  \Pcal_j^{f,q} (h_j(\zeta_j,\cdot))\big|_{\zeta_j=0} : 
       T_0\c^{N_j}\cong\c^{N_j}\to\c^n\ \  \text{is surjective for every $q\in Q$.}
\end{equation}
We extend each function $g_{j,k}$ by $0$ to $|\Bcal|\setminus C_j$, approximate $g_{j,k}\colon |\Bcal|\to\c$ by a 
holomorphic function $\wt g_{j,k}\in\Oscr(M)$, and set
\[
         \Xi(\zeta,p)=\prod_{j=1}^l\prod_{k=1}^{N_j}\big(1+\zeta_{j,k}\, \wt g_{j,k}(p)\big),\quad 
         \zeta=(\zeta_1,\ldots,\zeta_l)\in \c^{N_1}\times\cdots \times \c^{N_l},\; p\in M.
\]
Set $N=\sum_{j=1}^l N_j \ge nl$ and identify $\c^N=\c^{N_1}\times\cdots \times \c^{N_l}$.
If the approximation of $g_{j,k}$ by $\wt g_{j,k}$ is close enough for each $(j,k)$, then \eqref{eq:diPjt} ensures that
\[
     \frac{\di}{\di \zeta} \Pcal_j^{f,q} (\Xi(\zeta,\cdot))\big|_{\zeta=0} : T_0\c^N\cong\c^N\to(\c^n)^l\quad 
     \text{is surjective for every $q\in Q$.}
\]
This concludes the proof of Lemma \ref{lem:existence-sprays}.
\end{proof}


\section{Proof of Theorem \ref{th:main1}}\label{sec:proof}

In this section we prove Theorem \ref{th:main1} as a consequence of the following approximation result for multiplier functions.
Recall that $I=[0,1]$.

\begin{theorem}\label{th:main}
Assume that $S=K\cup\Gamma$ is an admissible subset of a connected open Riemann surface $M$
(see Definition \ref{def:admissible}), 
$\theta$ is a nowhere vanishing holomorphic $1$-form on $M$, and $n\in\n$ is  an integer. 
Let $f_t\colon M\to \c^n$ $(t\in I)$ be a continuous family of  full holomorphic maps and 
$\qgot_t\colon H_1(M; \z)\to\c^n$  be a continuous family of group homomorphisms.
Then, every continuous family of functions $\varphi_t\colon S\to\c_*$ $(t\in I)$ of class $\Ascr(S)$ such that
\[
       \int_\gamma \varphi_tf_t\theta=\qgot_t(\gamma)\quad \text{for every closed curve $\gamma\subset S$ and $t\in I$}
\]
may be approximated uniformly on $I\times S$ by continuous families of holomorphic functions $\wt \varphi_t\colon M\to\c_*$ 
such that
\[
       \int_\gamma \wt \varphi_tf_t\theta=\qgot_t(\gamma)\quad \text{for every closed curve $\gamma\subset M$ and $t\in I$}.
\]
Furthermore, if $\varphi_0$ extends to a holomorphic function $M\to\c_*$ such that $\int_\gamma \varphi_0f_0\theta=\qgot_0(\gamma)$ 
for all closed curves $\gamma\subset M$, then the homotopy $\wt \varphi_t$ can be chosen with $\wt \varphi_0=\varphi_0$.
\end{theorem}

The proof of Theorem \ref{th:main} consists of an inductive procedure; the following two lemmas provide the inductive step of the construction. 

%
%
\begin{lemma}[The noncritical case] \label{lem:noncritical}
Let $M$, $S=K\cup\Gamma\subset M$, $\theta$, $n$, and $f_t$ $(t\in I=[0,1])$ be as in Theorem \ref{th:main}.
Also let $L$ be a compact, smoothly bounded, Runge domain in $M$ such that $S\subset \mathring L$ and 
$S$ is a deformation retract of $L$.
Then, every continuous family of functions $\varphi_t\colon S\to\c_*$ $(t\in I)$ of class $\Ascr(S)$ may be  approximated
uniformly on $I\times S$ by continuous families of functions $\wt \varphi_t\colon L\to\c_*$ $(t\in I)$ of class $\Ascr(L)$ such that 
$(\wt \varphi_t-\varphi_t)f_t\theta$ is exact on $S$ for all $t\in I$. Furthermore, if $\varphi_0$ is of class $\Ascr(L)$ then the 
homotopy $\wt \varphi_t$ can be chosen with $\wt\varphi_0=\varphi_0$.
\end{lemma}

\begin{proof}
We may assume without loss of generality that $S$ is connected, hence so is $L$; otherwise we apply the same argument to 
each connected component. Let $l\in\z_+$ denote the rank of $H_1(S;\z)$. 

Let $f\colon I\times M\to\c^n$ and $\varphi\colon I\times S\to \c_*$ be the continuous maps defined by $f(t,\cdot)=f_t$ and 
$\varphi(t,\cdot)=\varphi_t$ for all $t\in I$. The assumptions on $f_t$ and $\varphi_t$ ensure that $\varphi_t f_t\in \Ascr(S)$ is full 
and nowhere flat on $\Gamma$ for all $t\in I$. Thus, Lemma \ref{lem:existence-sprays} furnishes finitely many holomorphic functions 
$g_1,\ldots, g_N\colon M\to\c$  such that the function $\Xi\colon \c^N\times M\to\c$ given by
\[
     \Xi(\zeta,p)=\prod_{i=1}^N \big(1+\zeta_i g_i(p)\big),\quad \zeta=(\zeta_1,\ldots,\zeta_N)\in\c^N,\; p\in M
\]
is a period dominating multiplier of $\varphi f$, in the sense that the period map
\begin{equation}\label{eq:period-vft}
    \c^N\ni\zeta\mapsto \Pcal^{\varphi f,t}(\Xi(\zeta,\cdot))\in (\c^n)^l
\end{equation}
has maximal rank equal to $ln$ at $\zeta=0$ for every $t\in I$. 
(Here, $\Pcal^{\varphi f,t}$ is the period map associated to a basis $\Bcal$ of $H_1(S;\z)$, 
the map $\varphi_t f_t$, and the $1$-form $\theta$; see \eqref{eq:periodmap} and \eqref{eq:periodmapj}.) 
In particular, since 
\begin{equation}\label{eq:Xivf=1}
     \Xi(0,\cdot)\equiv 1,
\end{equation} 
the implicit function theorem guarantees that the range of the period map \eqref{eq:period-vft} restricted 
to any ball $W$ around the origin in $\c^N$ contains an open neighborhood of 
\begin{equation}\label{eq:Xivf=12}
      \Pcal^{\varphi f,t}(\Xi(0,t))=\Pcal^{\varphi f,t}(1)\in(\c^n)^l,\quad  t\in I.
\end{equation}
Let $W$ be a small such ball satisfying $\Xi(\zeta,p)\neq 0$ for all $\zeta\in W$ and $p\in L$; such exists by 
\eqref{eq:Xivf=1} and compactness of $L$.

By the parametric version of Mergelyan's theorem we can approximate $\varphi$ uniformly on $I\times S$ by functions 
$\phi \colon I\times L\to\c_*$ of class $\Ascr(I\times L)$; recall that $S$ is a deformation retract of $L$
and $\c_*$ is an Oka manifold. Set $\phi_t=\phi(t,\cdot)$ for all $t\in I$. 
Furthermore, if $\varphi_0$ is of class $\Ascr(L)$ then the homotopy $\phi_t$ $(t\in I)$ can be chosen with 
$\phi_0=\varphi_0$. Thus, if the approximation of $\varphi$ by $\phi$ is close enough,  
the implicit function theorem furnishes a continuous path $\beta\colon I\to W$ such that 
\[
     \Pcal^{\phi f,t}(\Xi(\beta(t),\cdot))= \Pcal^{\varphi f,t}(1)\quad \text{for every $t\in I$}.
\]
If furthermore $\varphi_0$ is of class $\Ascr(L)$ and $\phi_0=\varphi_0$, then the path $\beta$ can be chosen such that
$\beta(0)=0\in W\subset\c^N$. It follows that the homotopy 
\[
	\wt\varphi_t:=\Xi(\beta(t),\cdot) \phi_t : L\to\c_*, \quad t\in I,
\]
satisfies the conclusion of the lemma provided that $W$ is chosen small enough and the approximation of $\varphi$ by $\phi$ is 
made sufficiently close.
\end{proof}

\begin{lemma}[The critical case] \label{lem:critical}
Let $M$, $\theta$, $n$, and $f_t$ $(t\in I=[0,1])$ be as in Theorem \ref{th:main}.
Let $\rho\colon M\to\r_+=[0,+\infty)$ be a smooth strongly subharmonic Morse exhaustion function and pick numbers $0<a<b\in\r$
which are not critical values of $\rho$ and such that $\rho$ has exactly one critical point $p$ in $L\setminus \mathring K$, 
where $K=\{\rho\leq a\}$ and $L=\{\rho\leq b\}$. 
Also let $\varphi_t\colon K\to\c_*$ and $\qgot_t\colon H_1(L;\z)\to \c^n$  $(t\in I)$ be continuous families of functions and group 
homomorphisms such that, for each $t\in I$, $\varphi_t$ is of class $\Ascr(K)$ and $\qgot_t(\gamma)=\int_\gamma \varphi_tf_t\theta$ 
holds for all closed curves $\gamma\subset K$. Then the family $\varphi_t$ may be  approximated uniformly on $I\times K$ by 
continuous families of functions $\wt \varphi_t\colon L\to\c_*$ $(t\in I)$ of class $\Ascr(L)$ such that 
\[
	\int_{\gamma} \wt \varphi_tf_t\theta=\qgot_t(\gamma) \ \ 
	\text{for every closed curve $\gamma\subset L$ and $t\in I$}.
\] 
Furthermore, if $\varphi_0$ is of class $\Ascr(L)$ and $\qgot_0(\gamma)=\int_\gamma \varphi_0f_0\theta$ for every closed curve
$\gamma\subset L$, then the homotopy $\wt \varphi_t\colon L\to\c_*$ $(t\in I)$ can be chosen such that $\wt\varphi_0=\varphi_0$.
\end{lemma}

\begin{proof}
By our assumptions, $p\in\mathring L\setminus K$ and the Morse index of $\rho$ at $p$ is either $0$ or $1$.

\smallskip

\noindent {\em Case 1:} The Morse index of $\rho$ at $p$ equals $0$. In this case a new connected 
component of the sublevel set $\{\rho\leq s\}$ appears when $s$ passes the value $\rho(p)$, and this gives a new 
connected and simply connected component $D$ of $L$. 
Since $K$ is a strong deformation retract of $L\setminus D$, Lemma \ref{lem:noncritical} provides a 
continuous family of functions $\wt \varphi_t\colon L\setminus D\to\c_*$ $(t\in I)$ of class $\Ascr(L\setminus D)$ which 
approximates the family $\varphi_t$ as closely as desired uniformly on $I\times K$
and such that $(\wt \varphi_t - \varphi_t)f_t\theta$ is exact on $K$ for every $t\in I$. Furthermore,
if $\varphi_0$ is of class $\Ascr(L)$ then the homotopy $\wt \varphi_t\colon L\setminus D\to\c_*$ can be chosen with 
$\wt\varphi_0=\varphi_0$.  Finally, define $\wt \varphi_t$ $(t\in I)$ on $D$ as any continuous family of maps of class 
$\Ascr(D,\c_*)$; if $\varphi_0$ is of class $\Ascr(L)$, we can choose $\wt \varphi_t=\varphi_0$ on $D$ for all $t\in I$. 
This concludes the proof in Case 1.

\smallskip

\noindent {\em Case 2:} The Morse index of $\rho$ at $p$ equals $1$. In this case there exists a smooth Jordan arc 
$\Gamma\subset \mathring L\setminus\mathring K$ with endpoints in $bK$ and otherwise disjoint from $K$ such that $S:=K\cup \Gamma$ 
is an admissible subset of $M$ (see Definition \ref{def:admissible}) and a strong deformation retract of $L$. If the endpoints of $\Gamma$ 
lie in different components of $K$ then the inclusion $K\hra S$ induces an isomorphism of the homology groups; otherwise 
there appears a new closed curve $\Gamma_0\subset S$, containing $\Gamma$, which is not in the homology of $K$. 
In each of these case we can use Lemma \ref{lem:periods} to construct a homotopy of continuous maps 
$h_t\colon \Gamma\to\c_*$ $(t\in I)$ such that the family of functions $\psi_t\colon S=K\cup\Gamma\to\c_*$, given by 
$\psi_t=\varphi_t$ on $K$ and $\psi_t=h_t$ on $\Gamma$, is continuous in $t\in I$ and, for each $t\in I$, $\psi_t$ is of class 
$\Ascr(S)$ and satisfies $\int_{\Gamma_0} \psi_t f_t\theta=\qgot_t(\Gamma_0)$ if $\Gamma$ lies in a closed curve $\Gamma_0$.
Furthermore, if $\varphi_0$ is of class $\Ascr(L)$ and $\qgot_0(\gamma)=\int_\gamma \varphi_0f_0\theta$ holds 
for every closed curve $\gamma\subset L$, then the homotopy $h_t\colon \Gamma\to\c_*$ $(t\in I)$ can be chosen with 
$h_0=\varphi_0|_\Gamma$. Lemma \ref{lem:noncritical}, applied to the homotopy $\psi_t\colon S\to\c_*$ $(t\in I)$, completes the task.
\end{proof}

%
%
\begin{proof}[Proof of Theorem \ref{th:main}]
Since the compact set $S$ is assumed to be Runge in $M$, there exist a smooth strongly subharmonic Morse exhaustion function 
$\rho\colon M\to\r$ and a number $a_1\in\r$ such that $a_1$ is a regular value of $\rho$, 
$S\subset\{\rho<a_1\}$, and $S$ is a strong deformation retract of 
$M_1:=\{\rho\le a_1\}$. Let $p_1,p_2,\ldots$ be the (isolated) critical points of $\rho$ in $M\setminus M_1$ and assume without loss of 
generality that $a_1<\rho(p_1)<\rho(p_2)<\cdots$. Choose a strictly increasing divergent sequence of real numbers $\{a_j\}_{j\ge 2}$ such that 
$\rho(p_{j-1})<a_j<\rho(p_j)$ for all $j=2,3,\ldots$ (in particular, $a_2>a_1$); if $\rho$ has only finitely many critical points in $M\setminus M_1$
then we choose the remainder terms of the sequence $\{a_j\}_{j\ge 2}$ arbitrarily. Set $M_0:=S$ and $M_j:=\{\rho\le a_j\}$ for all $j\ge 2$. 
Thus, each $M_j$ for $j\in\n$ is a smoothly bounded compact Runge domain in $M$ and we have that
\[
	S=M_0\Subset M_1\Subset M_2\Subset \cdots \Subset \bigcup_{j\in\z_+} M_j=M.
\]

Let $\varphi_t^0:= \varphi_t \colon S\to\c_*$ $(t\in I)$ be a continuous family of functions of class $\Ascr(S)$. Pick a number $\epsilon>0$. 
An inductive application of Lemmas \ref{lem:noncritical} and \ref{lem:critical} provides a sequence of continuous families 
$\varphi_t^j\colon M_j\to\c_*$ $(t\in I)$ of functions of class $\Ascr(M_j)$, $j\in\n$, such that the following conditions hold 
for each $t\in I$ and $j\in\n$:
\begin{enumerate}[\rm (a)]
\item $\varphi_t^j$ approximates $\varphi_t^{j-1}$ as closely as desired uniformly on $I\times M_{j-1}$,
\vspace{1mm}
\item $\int_{\gamma} \varphi_t^j f_t\theta=\qgot_t(\gamma)$ holds for every closed curve $\gamma\subset M_j$, and
\vspace{1mm}
\item  if $\varphi_0^0$ extends to a holomorphic function $M\to\c_*$ such that $\int_\gamma \varphi_0^0f_0\theta=\qgot_0(\gamma)$ 
holds for all closed curves $\gamma\subset M$, then the homotopy $\varphi_t^j$ can be chosen with $\varphi_0^j=\varphi_0^0$.
\end{enumerate}
If the approximation in {\rm (a)} is close enough for every $j\in\N$, we obtain a limit continuous family of holomorphic functions 
\[
	\wt \varphi_t:=\lim_{j\to\infty} \varphi_t^j : M\to\c \quad (t\in I)
\]
such that, for each $t\in I$, $\wt \varphi_t$ does not 
vanish anywhere on $M$, $\wt\varphi_t$ is uniformly $\epsilon$-close to $\varphi_t=\varphi_t^0$ on $S=M_0$, and 
$\int_{\gamma} \wt \varphi_tf_t\theta=\qgot_t(\gamma)$ holds for every closed curve $\gamma\subset M$. Furthermore, 
if $\varphi_0=\varphi_0^0$ extends to a holomorphic function $M\to\c_*$ such that $\int_\gamma \varphi_0f_0\theta=\qgot_0(\gamma)$ 
for all closed curves $\gamma\subset M$, then the homotopy $\wt\varphi_t$ $(t\in I)$ can be chosen 
such that $\wt\varphi_0=\varphi_0$. This concludes the proof of Theorem \ref{th:main}. 
\end{proof}

%
%
\begin{proof}[Proof of Theorem \ref{th:main1}]
Let $M$, $n$, $\Phi_t$, and $\qgot_t$ be as in Theorem \ref{th:main1}. Choose a nowhere vanishing holomorphic $1$-form $\theta$ on 
$M$ and set $f_t = \Phi_t/\theta \colon M\to\c^n$ for each $t\in I$. Let $S\subset M$ be a compact, smoothly bounded, simply connected 
domain, and consider the constant map $\varphi_t \equiv 1\in\c_*$ on $S$, $t\in I$. Theorem \ref{th:main} applied to 
these data provides a homotopy of holomorphic functions $h_t:=\wt\varphi_t\colon M\to\c_*$ $(t\in [0,1])$ satisfying the conclusion 
of Theorem \ref{th:main1}.
\end{proof}


\section{Applications of Theorem \ref{th:main1} to minimal surfaces} \label{sec:finalsec}

In this section we show how Theorem \ref{th:main1} can be used 
to prove Theorems \ref{th:Gauss-n}, \ref{th:avoiding} and Corollary \ref{cor:Crelle-sameGaussmap}.
The other results stated in the Introduction follow from these as has already been indicated. 

Recall that $\pi\colon \C^n_*=\c^n\setminus\{0\} \to\cp^{n-1}$ denotes the canonical projection onto the projective
space, so $\pi(z_1,\ldots,z_n)=[z_1\colon \cdots\colon z_n]$ are homogeneous coordinates on $\cp^{n-1}$. 
We shall need the following lemma concerning the lifting of holomorphic maps with respect to this projection.

%
%

\begin{lemma}\label{lem:lifting}
Let $M$ be an open Riemann surface and $Q\subset P$ be compact Hausdorff spaces with
$Q$ a strong deformation retract of $P$. (If $Q$ is empty, we assume that $P$ is contractible.)
Given a continuous map $g\colon M \times P \to \cp^{n-1}$
such that $g(\cdotp,p)\colon M \to \cp^{n-1}$ is holomorphic for every $p\in P$ and a 
continuous map $f\colon M\times Q\to \c^n_*$ such that $\pi\circ f=g|_{M\times Q}$ and 
$f(\cdotp,p)\colon M\to \c^n_*$ is holomorphic for every $p \in Q$, there exists 
a continuous map $\wt f\colon M\times P\to\c^n_*$ satisfying the following conditions:
\begin{itemize}
\item[\rm (a)] $\pi\circ \wt f =g$,
\vspace{1mm}
\item[\rm (b)] $\wt f=f$ on $M\times Q$, and 
\vspace{1mm}
\item[\rm (c)] $\wt f(\cdotp,p)\colon M\to\c^n_*$ is holomorphic for every $p\in P$.
\end{itemize}
\end{lemma}

\begin{proof}
Note that the map $\pi\colon \c^n_*\to\cp^{n-1}$ is a holomorphic fiber bundle with fiber $\c_*$
which is an Oka manifold. For such bundles, the parametric Oka principle for liftings holds
for maps from any reduced Stein space, in particular, for maps from an open Riemann surface
(see \cite[Theorem 4.2]{Forstneric2010TM}). In our situation this means the following:

Given a continuous map $\wt f\colon M\times P\to\c^n_*$ satisfying conditions (a) and (b) above,
there is a homotopy $\wt f_t\colon  M\times P\to\c^n_*$ $(t\in[0,1])$ such that $\wt f_0=\wt f$,
every map $\wt f_t$ in the family enjoys conditions (a) and (b), and the final map $\wt f_1$ also 
satisfies condition (c).

This reduces the proof to the existence of a continuous map $\wt f\colon M\times P\to\c^n_*$ satisfying 
conditions (a) and (b) (but not necessarily condition (c)).  
Let $\pi\colon E\to \cp^{n-1}$ denote the holomorphic line obtained 
by adding the zero section $E_0\cong \cp^{n-1}$ to the $\c_*$-bundle  $\pi\colon \c^n_*\to\cp^{n-1}$.
Since $M$ is homotopy equivalent to a wedge of circles, the pullback $g^*E\to M$
by any map $g\colon M\to \cp^{n-1}$ is a trivial complex line bundle over $M$, and hence it admits
a nowhere vanishing section. Clearly, such a section corresponds to a lifting $\wt f\colon M\to \c^{n}_*$ 
of the map $g$. Furthermore, if $P$ is a contractible compact Hausdorff space then
by the same argument a map $g\colon M\times P\to \cp^{n-1}$ lifts to a map
$\wt f\colon M\times P \to \c^{n}_*$. Similarly, if $Q\subset P$ is a nonempty subspace such that 
$P$ deformation retracts onto $Q$ and we already have a lifting $\wt f$ of $g$ over $M\times Q$,
then $\wt f$ extends to a lifting $\wt f\colon M\times P\to \c^n_*$ extending $g$. 
This completes the proof. 
\end{proof}

We may assume in the sequel that the Riemann surface $M$ is connected; otherwise we can apply
the same proofs separately to each connected component.

%
%
\begin{proof}[Proof of Theorem \ref{th:Gauss-n}]
Assume first that the map $\Gscr\colon M\to Q_{n-2} \subset \CP^{n-1}$ is full. 
Let $\pgot \colon H_1(M;\Z)\to \R^n$ be any  group homomorphism. 
By Lemma \ref{lem:lifting} there is a holomorphic lifting $f\colon M\to\Agot_*$ of $\Gscr$, 
where $\Agot_*=\Agot\setminus\{0\}\subset\c^n$ is the punctured null quadric \eqref{eq:nullquadric}.
Pick a holomorphic $1$-form $\theta$ without zeros on $M$; such exists by the Oka-Grauert principle. 
Let $\qgot\colon H_1(M;\Z)\to \C^n$ be the group homomorphism given by 
\[
	\qgot(\gamma) = \int_\gamma f\theta \quad \text{for every closed curve $\gamma\in H_1(M;\Z)$}.
\]
Choose a homotopy of group homomorphisms $\qgot_t\colon  H_1(M;\Z)\to \C^n$ $(t\in [0,1])$ such that
$\qgot_0=\qgot$ and $\qgot_1=\igot \, \pgot$. If $\qgot=\qgot^1+\imath \qgot^2$
with $\qgot^1,\qgot^2\colon  H_1(M;\Z)\to \R^n$, we can take 
\[
	\qgot_t = (1-t)\qgot^1 +\imath \left((1-t)\qgot^2 + t\, \pgot\right),\quad t\in  [0,1].
\]
Theorem \ref{th:main1}, applied to the $1$-form $\Phi=f\theta$ 
and the homotopy of group homomorphisms $\qgot_t$, furnishes a nowhere vanishing holomorphic function 
$h\colon M\to \C_*$ such that the $1$-form $h f\theta$ has periods equal to $\igot\,\pgot$.
In particular, its real part is exact and hence it integrates to a conformal minimal immersion
$X\colon M\to\R^n$ with $\Flux_X=\pgot$ by setting
\[
	X(p)= 2\int_{p_0}^p \Re\left( hf\theta \right),\quad p\in M
\]
for any initial point $p_0\in M$. The Gauss map of $X$ equals $[\di X]=[hf\theta]=[f]=\Gscr$. 
If $\pgot=0$ then $X$ is the real part of a holomorphic null curve  $X+\imath Y\colon M\to \C^n$. 

In order to prove that $X$ can be chosen an embedding
if $n\ge 5$ and an immersion with simple double points if $n=4$, we proceed as in 
\cite[proof of Theorem 4.1]{AlarconForstnericLopez2016MZ} (see also \cite[Section 6]{AlarconForstneric2014IM}), 
with the only difference that we use Lemma \ref{lem:existence-sprays} from the present paper in order to make 
a generic perturbation of the integral $\int_\gamma \Re(hf\theta)$ along an arc $\gamma\subset M$ connecting 
a given pair of points $p,q\in M$. We leave out the obvious details.

If the map $\Gscr$ is not full, we can apply the same proof with $\C^n$ replaced by the $\C$-linear span 
$\Lambda= \span(f(M)) \subset \C^n$ of the image of the lifted map $f \colon M\to \Agot_*$. 
Note that $f$ is full in $\Lambda$, so the same proof applies and gives a conformal minimal immersion 
$X\colon M\to \R^n$ with the generalized Gauss map $\Gscr$ and with $\Flux_X$ being any homomorphism 
$\pgot\colon H_1(M;\Z)\to\R^n$ such that $\qgot=\imath\, \pgot$ has range in $\Lambda$. 
If $\Lambda$ is a complex line, the result also follows from the Gunning-Narasimhan theorem
\cite{GunningNarasimhan1967MA}.  The general position theorem still applies and shows that 
$X$ can be chosen an embedding  if $\dim \Lambda\ge 5$ and an immersion with simple double points 
if $\dim \Lambda=4$. 
\end{proof}

In the proof of Theorem \ref{th:avoiding} we shall need the following lemma.

%
%

\begin{lemma}\label{lem:2points}
Let $M$ be an open Riemann surface. For any holomorphic map $g \colon M\to \cp^1$ there is a homotopy 
of holomorphic maps $g_t \colon M\to \cp^1$  $(t\in [0,1])$ such that $g_0=g$, $g_t$ is nonconstant for every 
$t\in (0,1]$, and $g_1(M)$ omits any two given points of the Riemann sphere. 
\end{lemma} 

Note that if $M$ equals $\C$ or $\C_*$ then a nonconstant Gauss map $g\colon M\to\CP^1$ 
cannot omit three points of the Riemann sphere in view of Picard's theorem.

\begin{proof}[Proof of Lemma \ref{lem:2points}]
Without loss of generality we may assume that $g\colon M\to\CP^1$ is a nonconstant  holomorphic map. 
Pick a pair of points  $a,b\in \cp^1$. The surface $M$ contains a 1-dimensional embedded CW-complex $C\subset M$ 
such that there is a strong deformation retraction $\rho_t\colon M\to M$ $(t\in [0,1])$, 
i.e. $\rho_0=\mathrm{Id}_M$, $\rho_t|_C= \mathrm{Id}_C$ for all $t\in [0,1]$, and $\rho_1(M)=C$. 
(Such a CW-complex $C\subset M$ representing the topology of $M$ can be obtained 
as the Morse complex of a Morse strongly subharmonic exhaustion function on $M$.)
By a small generic deformation of $C$ we may assume that $g(C)\subset \CP^1\setminus \{a,b\}$.
Consider the homotopy of continuous maps $h_t=g \circ \rho_t\colon M\to \cp^1$ for $t\in [0,1]$. 
Clearly, $h_0=g$ and $h_1=g\circ \rho_1$; hence $h_1(M)=g(C)\subset \CP^1 \setminus \{a,b\}$.
Since $\CP^1\setminus \{a,b\}\cong \c_*$ is an Oka manifold, there is a homotopy 
$h_t\colon M\to \CP^1\setminus \{a,b\}$ $(t\in [1,2])$ connecting the continuous map $h_1$ to a holomorphic map 
$h_2\colon M\to \CP^1\setminus \{a,b\}$. Clearly, we can arrange  by a generic deformation that $h_2$ is nonconstant.

Pick a pair of points $p,q\in M$ such that $h_2(p)\ne h_2(q)$. 
By general position  we may assume that $h_t(p)\ne h_t(q)$ for all $t\in (0,2]$.
(Note that the maps $h_t$ for $t\in (0,2)$ are merely continuous, so it is trivial to satisfy this condition.)
Since $\cp^1$ is an Oka manifold, we can apply the 1-parametric Oka property with interpolation
on the pair of points $\{p,q\}\subset M$ in order to deform the homotopy $(h_t)_{t\in [0,2]}$ with fixed ends 
$h_0$ and $h_2$ to a homotopy $(g_t)_{t\in [0,2]}$ consisting of holomorphic maps $g_t\colon M\to \cp^1$
such that $g_t(p)=h_t(p)$ and $g_t(q)=h_t(q)$ for all $t\in [0,2]$
(see \cite[Theorem 5.4.4]{Forstneric2017E}). In particular, $g_0=g$ and the map $g_t$ is nonconstant for each 
$t\in (0,2]$. To conclude the proof we reparametrize the interval $[0,2]$ of the homotopy back to $[0,1]$.
\end{proof}

%
%

\begin{proof}[Proof of Theorem \ref{th:avoiding}]
Let $X\colon M\to\r^3$ be a conformal minimal immersion. Denote by $g\colon M\to \CP^1$ its complex Gauss map
\eqref{eq:C-Gauss}. To simplify the notation, we identify $\CP^1$ with the quadric $Q_1\subset\CP^2$
\eqref{eq:nullquadric-projected}, so $g$ is obtained from the generalized Gauss map $G_X=[\di X]$ 
by the formula \eqref{eq:C-Gauss}. Fix a nowhere vanishing holomorphic $1$-form $\theta$ on $M$ and let 
$f=\di X/\theta\colon M\to \Agot_*$ (see \eqref{eq:nullquadric}). By taking into account the above identification, we shall 
write $\pi\circ f=g$.

Let $a,b\in \CP^1$ be any pair of points. 
By Lemma \ref{lem:2points} there is a homotopy of holomorphic maps $g_t \colon M\to \cp^1$  $(t\in [0,1])$ 
such that $g_0=g$, $g_t$ is nonconstant for every $t\in (0,1]$, and $g_1(M)\subset \CP^1\setminus\{a,b\}$.
By Lemma \ref{lem:lifting}, applied with $Q=\{0\}\subset P=[0,1]$,  
there is a homotopy of holomorphic maps $f_t\colon M\to \Agot_*$ such that
$f_0=f$ and $\pi\circ f_t=g_t$ for every  $t\in [0,1]$. By Theorem \ref{th:main1} there is a homotopy
of holomorphic multipliers $h_t\colon M\to \C_*$ such that $h_0=1$ and the real $1$-form 
$\Re(h_tf_t\theta)$ is exact for every $t\in [0,1]$. (We can also arrange that the complex $1$-form
$h_1f_1\theta$ is exact.) Fix a point $p_0\in M$. Then, for any $t\in [0,1]$ the map $X_t\colon M\to\R^n$ 
given by
\[
	X_t(p) = X(p_0)+2\int_{p_0}^p \Re(h_tf_t\theta),\quad p\in M
\] 
is a conformal minimal immersion with the complex Gauss map $\pi(h_tf_t)=\pi(f_t)=g_t$, and we also
have $X_0=X$ since $h_0=1$. This proves the first part of the theorem. 

To prove the second part, we choose the homotopies $g_t$ and $X_t$ as above such that $g_1(M)\subset \C_*$
and $\Flux_{X_1}=0$.  To simplify the notation, we drop the index $1$ and simply write $X$ and $g$.
To complete the proof, it remains to find an isotopy of conformal minimal immersions connecting 
$X=(X_1,X_2,X_3)$ to a flat immersion.

Since $X$ has vanishing flux, the holomorphic $1$-form $\di X=(\di X_1,\di X_2,\di X_3)$ 
is exact. Set $\phi_3=\di X_{3}$.  From the Weierstrass representation \eqref{eq:EWR} 
we see  the holomorphic $1$-forms $\phi_3$, $g\phi_3$ and $g^{-1}\phi_3$ are exact since they are linear
combinations with constant coefficients of the components of $\di X$. 
Consider the $1$-parameter family of holomorphic $1$-forms
\[
	\Phi_\lambda= \left( \frac{1}{2} \left(\frac1{g}- \lambda^2 g\right), 
	\frac{\imath}{2} \left(\frac1{g}+\lambda^2 g\right),\lambda \right) \phi_3, \quad \lambda\in \C.
\]
Note that $\Phi_\lambda$ is nowhere vanishing and exact for every $\lambda$, 
$\Phi_1= \di X$, $\Phi_\lambda/\theta$ has values in $\Agot_*$, and  $\Phi_0=\left(\frac12,\frac{\imath}2,0\right)\frac{\phi_3}{g}$ 
is clearly flat. Therefore, for every $\lambda\in \C$ the $1$-form $\Phi_\lambda$ integrates to a holomorphic null curve 
$Z_\lambda(p)=X(p_0)+ 2 \int_{p_0}^p \Phi_\lambda$ $(p\in M)$. The family of
conformal minimal immersions $X_\lambda=\Re Z_\lambda\colon M\to\R^3$ for $\lambda\in [0,1]$
then connects $X_1=X$ to the flat immersion $X_0$. 
\end{proof}

\begin{proof}[Proof of Corollary \ref{cor:Crelle-sameGaussmap}]
Let $\pgot=\Flux_X\colon H_1(M;\Z)\to\R^n$, and let $\pgot' \colon H_1(M;\Z)\to\R^n$ be any group homomorphism.
Assume first that $X$ is full. Fix a point $p_0\in M$. Applying Theorem \ref{th:main1} to the isotopy of homomorphism 
\begin{equation}\label{eq:qgot-t}
	\qgot_t=\imath (t\, \pgot' + (1-t)\pgot) : H_1(M;\Z)\to \imath\,\r^n
\end{equation}
and $1$-forms $\Phi_t=2\di X$ we obtain a homotopy of conformal minimal immersions 
$X_t\colon M\to\R^n$ $(t\in [0,1])$ of the form
\[
	X_t(p)=X(p_0) +  2 \int_{p_0}^p \Re( h_t \,\di X), \quad p\in M
\]
with $\Flux_{X_t}=t\pgot' + (1-t) \pgot$ $(t\in [0,1])$ and with $h_0=1$, so $X_0=X$. Then, $\Flux_{X_1}=\pgot'$.
By choosing $\pgot'=0$ we get $\Flux_{X_1}=0$ and hence $X_1$ is the real part of a holomorphic null curve $M\to \C^n$. 
Note that the generalized Gauss map of $X_t$ equals $[h_t \,\di X] =[\di X]=G_X$ and hence is independent of $t\in [0,1]$.
If $X$ is not full, we can apply the same proof with $\C^n$ replaced by the $\C$-linear span 
$\Lambda\subset \C^n$ of the image of the map $\di X/\theta\colon M\to \Agot_*\subset \C^n$, assuming 
that the range of $\qgot_t$ \eqref{eq:qgot-t} belongs to $\Lambda$ for every $t\in [0,1]$. 
If $\pgot'=0$, this holds if and only if $\imath\, \pgot$ has range in $\Lambda$.
\end{proof}

%
%

\section{A structure theorem}\label{sec:structure} 

In light of Theorem \ref{th:Gauss-n}, it is a natural problem to describe 
the space of all conformal minimal immersions with the same generalized Gauss map. 
In this section, we prove that the space of all holomorphic null curves
from a compact bordered Riemann surface to $\C^n$ with a given generalized Gauss map 
is a complex Banach manifold (see Corollary \ref{co:structure}); if we consider
instead conformal minimal immersions $M\to\R^n$ (also with prescribed flux map), 
then we get a real analytic Banach manifold (see Corollary \ref{co:structure-minimal}). 
These results are in the spirit of \cite[Theorem 3.1]{AlarconForstnericLopez2016MZ}.

Recall that a {\em compact bordered Riemann surface} is a compact Riemann surface $M$ with nonempty 
boundary $\emptyset\neq bM\subset M$ consisting of finitely many pairwise disjoint smooth Jordan curves. 
The interior $\mathring M=M\setminus bM$ of such $M$ is called a {\em bordered Riemann surface}. 
Every compact bordered Riemann surface $M$ is diffeomorphic to a smoothly bounded, compact domain 
in an open Riemann surface $\wt M$ (see e.g.\ Stout \cite{Stout1965TAMS}). 

We begin with the following technical result concerning the derivative maps.

\begin{theorem}\label{th:structure}
Let $M$ be a compact bordered Riemann surface, let $\theta$ be a  nowhere vanishing holomorphic $1$-form on $M$, 
and let $f\colon M\to\c^n_*$ be a map of class $\Ascr(M)$. Then the following hold:
\begin{enumerate}[\it i)]
\item For any $r\in\z_+$ and group homomorphism $\qgot\colon H_1(M;\z)\to\span(f(M))\subset\c^n$ 
the space of all functions $h\in\Ascr^r(M,\c_*)$ satisfying 
\[
    \int_\gamma hf\theta=\qgot(\gamma)\quad \text{for every closed curve $\gamma\subset M$}
\]
is a complex Banach manifold with the natural $\Cscr^r(M)$-topology.
\vspace{1mm}
\item For any $r\in\z_+$ and group homomorphism $\qgot\colon H_1(M;\z)\to\span(f(M))\subset\c^n$ 
the space of all functions $h\in\Ascr^r(M,\c_*)$ satisfying 
\[
    \int_\gamma \Re(hf\theta)=\Re(\qgot(\gamma))\quad \text{for every closed curve $\gamma\subset M$}
\]
is a real analytic Banach manifold with the natural $\Cscr^r(M)$-topology.
\end{enumerate} 
\end{theorem}

\begin{proof}
Set $l=\dim H_1(M;\z)$, $\Sigma=\span(f(M))\subset\c^n$, and $n^*=\dim(\Sigma)\le n$. 
Let $\Pcal\colon\Ascr(M)\to(\c^n)^l$ be the period map associated to a fixed basis $\Bcal$ of $H_1(M;\z)=\Z^l$, $f$, and $\theta$ 
(see \eqref{eq:periodmap}). 

If $M$ is simply connected then $l=0$ and the theorem is trivial. Indeed, in this case the period conditions 
are void and hence {\it i)} and {\it ii)} hold since $\Ascr^r(M,\c_*)$ is a complex Banach manifold 
(see \cite[Theorem 1.1]{Forstneric2007AJM}).

Assume now that $l>0$. Pick an integer $r\in\z_+$ and a group homomorphism $\qgot\colon H_1(M;\z)\to\Sigma$. 
Denote by $\Ascr^r_\qgot(M,\c_*)$ (resp. $\Ascr^r_{\Re\qgot}(M,\c_*)$) the set of all functions $h\in\Ascr^r(M,\c_*)$ satisfying 
$\int_\gamma hf\theta=\qgot(\gamma)$ (resp.\ $ \int_\gamma \Re(hf\theta)=\Re(\qgot(\gamma))$) for all closed curves 
$\gamma\subset M$. By Lemma \ref{lem:existence-sprays}, the differential $d\Pcal_{h_0}$ of the restricted period map 
$\Pcal\colon \Ascr^r(M,\c_*)\to \Sigma^l$ at any point $h_0\in\Ascr^r(M,\c_*)$ has maximal rank equal to $ln^*$.
Thus, the implicit function theorem ensures that $h_0$ admits an open neighborhood $\Omega\subset\Ascr^r(M,\c_*)$ 
such that $\Omega\cap \Ascr^r_\qgot(M,\c_*)$ is a complex Banach submanifold of $\Omega$ which is parametrized 
by the kernel of the differential $d\Pcal_{h_0}$ of $\Pcal$ at $h_0$; this is a complex codimension $ln^*$ subspace of the 
complex Banach space $\Ascr^r(M,\c)$ (the tangent space of $\Ascr^r(M,\c_*)$). Likewise, 
$\Omega\cap \Ascr^r_{\Re\qgot}(M,\c_*)$ is a real analytic Banach submanifold of $\Omega$ which is parametrized by 
the kernel of the real part $\Re(d\Pcal_{h_0})$ of $d\Pcal_{h_0}$. This proves {\it i)} (resp. {\it ii)}).
\end{proof}

\begin{corollary}\label{co:structure}
Let $M$ and $f$ be as in Theorem \ref{th:structure}.  For any integer $r\ge 1$ the space of holomorphic immersions 
$F\colon M\to\c^n$ of class $\Ascr^r(M)$ with the generalized Gauss map $G_F=\pi\circ f\colon M\to\cp^{n-1}$ is a 
complex Banach manifold.
\end{corollary}

\begin{proof}
Let $\theta$ be a holomorphic $1$-form vanishing nowhere on $M$. By Theorem \ref{th:structure} {\it i)}, applied to the integer $r-1\in\z_+$ 
and the group homomorphism $\qgot\equiv 0$, the space $\Ascr^{r-1}_\qgot(M,\c_*)$ of all functions $h\in\Ascr^{r-1}(M,\c_*)$ such that 
$hf\theta$ is exact on $M$ is a complex Banach manifold with the natural $\Cscr^{r-1}(M)$-topology. Fixing $p_0\in\mathring M$, the 
integration $M\ni p\mapsto z+\int_{p_0}^p hf\theta$, with an arbitrary choice of the initial value $z\in\c^n$, provides an isomorphism 
between the Banach manifold $\Ascr^{r-1}_\qgot(M,\c_*)\times\c^n$ and the space of holomorphic immersions $M\to\c^n$ of class 
$\Ascr^r(M)$ with the generalized Gauss map $\pi\circ f$; hence the latter is also a complex Banach manifold.
\end{proof}

\begin{corollary}\label{co:structure-minimal}
Let $M$ be a compact bordered Riemann surface and $f\colon M\to\Agot_*$ be a map of class $\Ascr(M)$, where $\Agot$ is the 
null quadric \eqref{eq:nullquadric}. Then the following hold: 
\begin{enumerate}[\it i)]
\item For any integer $r\ge 1$ the space of conformal minimal immersions $X\colon M\to\r^n$ of class $\Cscr^r(M)$ with 
the generalized Gauss map $G_X=\pi\circ f\colon M\to\cp^{n-1}$ is a real analytic Banach manifold with the natural $\Cscr^r(M)$-topology.
\vspace{1mm}
\item For any integer $r\ge 1$ and any group homomorphism 
$\qgot\colon H_1(M;\z)\to\span(f(M))\cap\{z\in\c^n\colon \Re(z)=0\}\subset\c^n$ the space of conformal minimal immersions 
$X\colon M\to\r^n$ of class $\Cscr^r(M)$ with the generalized Gauss map $G_X=\pi\circ f\colon M\to\cp^{n-1}$ and the flux map 
$\Flux_X=\imath\qgot\colon H_1(M;\z)\to\r^n$ is real analytic Banach manifold.
\end{enumerate} 
\end{corollary}

\begin{proof}
Let $\theta$ be a holomorphic $1$-form vanishing nowhere on $M$. By Theorem \ref{th:structure} {\it ii)}, applied to the integer $r-1\in\z_+$ 
and the group homomorphism $\qgot\equiv 0$, the space $\Ascr^{r-1}_{\Re\qgot}(M,\c_*)$ of all functions $h\in\Ascr^{r-1}(M,\c_*)$ such that 
$\Re(hf\theta)$ is exact on $M$ is a real analytic Banach manifold with the natural $\Cscr^{r-1}(M)$-topology. Fixing $p_0\in\mathring M$, 
the integration $M\ni p\mapsto x+\int_{p_0}^p \Re(hf\theta)$, with an arbitrary choice of the initial value $x\in\r^n$, provides an isomorphism 
between the Banach manifold $\Ascr^{r-1}_{\Re\qgot}(M,\c_*)\times\r^n$ and the space of conformal minimal immersions $M\to\r^n$ of 
class $\Ascr^r(M)$ with the generalized Gauss map $\pi\circ f$, and so the latter is also a Banach manifold. This proves {\it i)}.

Assertion {\it ii)} follows from the same argument applied to the group homomorphism $-\qgot$ and using Theorem \ref{th:structure} {\it i)} 
instead of Theorem \ref{th:structure} {\it ii)}.
\end{proof}


\section{Path components of the space of conformal minimal immersions $M\to\r^n$} 
\label{sec:path}

Assume that $M$ is an open connected Riemann surface and $n\ge 3$ is an integer.
Recall that a conformal minimal immersion $X\colon M\to\r^n$ is said to be {\em flat} if its image $X(M)$ 
lies in an affine $2$-plane of $\r^n$; otherwise it is {\em nonflat}. Let us denote by $\Mgot(M,\R^n)$ the space of all conformal 
minimal immersions $M\to\R^n$ endowed with the compact-open topology, and let $\Mgot_*(M,\R^n)$ denote 
the open subset of $\Mgot(M,\R^n)$ consisting of all nonflat immersions. Fix a nowhere vanishing holomorphic 
$1$-form $\theta$ on $M$ and consider the maps 
\[
	\Mgot(M,\R^n) \longrightarrow \Oscr(M, \Agot_* ) \longhookrightarrow \Cscr(M,\Agot_*),
\]
where $\Agot_*=\Agot_*^{n-1}\subset\C^n$ is the punctured null quadric \eqref{eq:nullquadric}, the first map above
is given by $X\mapsto \di X/\theta$, and the second map is the natural inclusion of the space of all holomorphic maps
$M\to \Agot_*$ into the space of all continuous maps. 

Since $\Agot_*$ is an Oka manifold, the inclusion 
$\Oscr(M, \Agot_* ) \hookrightarrow \Cscr(M,\Agot_*)$ is a weak homotopy equivalence by the main result of 
Oka theory (see \cite[Chapter 5]{Forstneric2017E}).
Forstneri\v c and L\'arusson proved in \cite{ForstnericLarusson2016} that the 
restricted map  $\Mgot_*(M,\R^n) \to \Oscr(M, \Agot_* )$, $X\mapsto \di X/\theta$, is also a weak homotopy equivalence.
(If the homology group $H_1(M;\Z)$ is finitely generated, then both these maps are actually homotopy equivalences, in fact, 
inclusions of deformation retracts; see \cite[Section 6]{ForstnericLarusson2016}.) 
Even more, 
the map $\Oscr(M, \Agot_* )\to H^1(M,\c^n)$ 
sending a map $f\in\Oscr(M, \Agot_* )$ to the cohomology class of $f\theta$ is a Serre fibration; 
see Alarc\'on and L\'arusson \cite{AlarconLarusson2017IJM}. It follows in particular that the
path connected components of $\Mgot_*(M,\R^n)$ are in bijective correspondence
with the path components of the space $\Cscr(M,\Agot_*^{n-1})$. Since $M$ is homotopy equivalent
to a bouquet of circles and we have $\pi_1(\Agot_*^{2})=H_1(\Agot^2_*;\Z)=\Z_2=\Z/2\Z$ and $\pi_1(\Agot_*^{n-1})=0$ if $n>3$,
it follows that the path components of $\Mgot_*(M,\R^3)$  are in bijective correspondence with
group homomorphisms $H_1(M;\Z)\to \Z_2$ (hence with elements of the 
abelian group $(\Z_2)^l$ where $l\in \Z^+\cup\{\infty\}$ denotes the number of generators
of $H_1(M;\Z)$), and $\Mgot_*(M,\R^n)$ is path connected if $n>3$ (see \cite[Corollary 1.4]{ForstnericLarusson2016}).

In this section we show the following result which also includes flat immersions.

\begin{theorem}\label{th:pathconnected}
Let $M$ be an open connected Riemann surface. The natural inclusion $\Mgot_*(M,\R^n)\hra \Mgot(M,\R^n)$
of the space of all nonflat conformal minimal immersions $M\to\R^n$ into the space of all conformal minimal immersions
induces a bijection of path components of the two spaces. In particular, the set of path components of
$\Mgot(M,\R^3)$ is in bijective correspondence with the elements of the abelian group $(\Z_2)^l$
where $H_1(M;\Z)=\Z^l$ $(l\in\Z_+\cup\{\infty\})$, and $\Mgot(M,\R^n)$ is path connected if $n>3$.
\end{theorem}

In view of \cite[Corollary 1.4]{ForstnericLarusson2016}, the case $n>3$ of Theorem \ref{th:pathconnected} trivially follows 
from the following result.

\begin{theorem}\label{th:nonflat}
Let $M$ be a connected open Riemann surface. Given a flat conformal minimal immersion $X\colon M\to\r^n$ $(n\ge 3)$, 
there exists an isotopy $X_t\colon M\to\r^n$ $(t\in[0,1])$ of conformal minimal immersions such that $X_0=X$ and $X_1$ is nonflat.
\end{theorem}

In dimension $n=3$, we obtain Theorem \ref{th:pathconnected}  by combining \cite[Corollary 1.4]{ForstnericLarusson2016} 
with the following result which shows that every homotopy class of maps $M\to \Agot_*^2$ contains the derivative
of a flat conformal minimal immersion $M\to\R^3$.

\begin{theorem}\label{th:flat-3}
Let $M$ be a connected open Riemann surface and $\theta$ be a nowhere vanishing holomorphic $1$-form on $M$.
For every group homomorphism $\pgot\colon H_1(M;\Z)\to \Z_2$ there exists a flat conformal minimal immersion 
$X\colon M\to\R^3$ satisfying $H_1(\di X/\theta) = \pgot$.
\end{theorem}

We begin with some preparations.  Set $I:=[0,1]$. For any continuous function $a\colon I\to\c$ we denote by 
$\Pscr^a\colon \Cscr(I,\c)\to\c^2$ the period map given by
\[
      \Pscr^a(f)=\int_0^1 a(s)(f(s),f(s)^2)\, ds,\quad f\in \Cscr(I,\c).
\]

\begin{lemma}\label{lem:gg2-spray}
Let $I'$ be a nontrivial closed subinterval of $I = [0,1]$ and let $f\colon I\to\c$ and $a\colon I\to\c_*$ be continuous functions 
such that $f$ is not constant on $I'$. There exist finitely many continuous functions $g_1,\ldots, g_N\colon I\to\c$ $(N \ge n)$, 
supported on $I'$, such that the function $h\colon \c^N\times I\to \c$ given by
\[
       h(\zeta,s):=\prod_{i=1}^N \big(1+\zeta_i g_i(s)\big),\quad \zeta=(\zeta_1,\ldots,\zeta_N)\in\c^N,\; s\in I
\]
is such that the map
\[
    \frac{\di}{\di\zeta} \Pscr^a(h(\zeta,\cdot)f)\big|_{\zeta=0} : T_0\c^N\cong\c^N\lra \c^2\quad \text{is surjective}.
\]
\end{lemma}

\begin{proof}
As in the proof of Lemma \ref{lem:spray-loops} we pick an integer $N\ge 2$ and continuous functions $g_1,\ldots,g_N\colon I\to\c$, 
supported on $I'$, which will be specified later, and define $h$ as in \eqref{eq:h-spray}. 
Let  $\zeta=(\zeta_1,\ldots,\zeta_N)$ be holomorphic coordinates in $\c^N$.
Consider the period map $\Psf\colon \c^N\to\c^2$ given by
\[
      \Psf(\zeta)=\Pscr^a(h(\zeta,\cdot)f)=\int_0^1 a(s) \left(h(\zeta,s)f(s),h(\zeta,s)^2f(s)^2\right) ds,\quad \zeta\in\c^N.
\]
Equation \eqref{eq:dih} gives
\[
     \left.\frac{\di \Psf(\zeta)}{\di\zeta_i}\right|_{\zeta=0}= \int_0^1 a(s)\left(g_i(s)f(s),2g_i(s)f(s)^2\right) ds,\quad i\in\{1,\ldots,N\},
\]
Since $f$ is continuous and nonconstant on the interval $I'$, there are points $s_1,\ldots,s_N\in \mathring I'$ such that  
\[
	\span \left\{(f(s_i),2f(s_i)^2) : i=1,\ldots,N\right\} = \c^2. 
\]
Reasoning as in the proof of Lemma \ref{lem:spray-loops}, taking 
into account that the function $a$ has no zeros on $I$, we conclude the proof by suitably choosing the functions $g_i$ with support in a 
small neighborhood of $s_i$ in $I'$.
\end{proof}

\begin{lemma}\label{lem:gg2-period}
Let $f\colon I\to \c$ and $a\colon I\to\c_*$ be continuous functions and assume that $f$ is not constant. Also let $x_1,x_2\in\c$ 
be complex numbers. Then there exists a continuous function $h\colon I\to \c_*$ such that $h(s)=1$ for $s\in  \{0,1\}$ and
\[
	\int_0^1 a(s)\left(h(s) f(s),h(s)^2f(s)^2\right) ds = (x_1,x_2).
\]
\end{lemma}

\begin{proof}
As in the proof of Lemma \ref{lem:periods}, and in view of Lemma \ref{lem:gg2-spray}, it suffices to prove that for any $\epsilon>0$ 
there exists a function $h\colon I\to \c_*$ such that
\begin{equation}\label{eq:approximate1}
	\left| \int_0^1 a(s)\left(h(s) f(s),h(s)^2f(s)^2\right) ds - (x_1,x_2)\right| <\epsilon, \quad t\in[0,1].
\end{equation}
To construct such a function $h$ we reason as follows. 
Since $a$ vanishes nowhere on $I$ and $f$ is continuous and nonconstant, there exist a big integer $N\in\n$ and numbers 
$0<s_1<\cdots<s_N<1$ such that the map $\c^N\to \c^2$ given by
\[
	(y_1,\ldots,y_N)\longmapsto \sum_{i=1}^N a(s_i)\left(y_if(s_i),y_i^2f(s_i)^2\right)
\] 
is surjective. Fix numbers $\tau>0$ and $0<\epsilon'<\epsilon$ which will be specified later, 
and choose numbers $y_1,\ldots,y_N\in\c_*$ such that 
\begin{equation}\label{eq:'1}
     \left|\sum_{i=1}^N a(s_i) \left(y_if(s_i),y_i^2f(s_i)^2\right)- (x_1,2\tau x_2\big)\right|<\epsilon'.
\end{equation}
Given a constant $\eta>0$ to be specified later, we let $h\colon I\to\c_*$ be a continuous function satisfying 
$h(0)=h(1)=1$  and also the following conditions:
\begin{enumerate}[\rm (a)]
\item $|h(s)|\le 1$ for $s\in[0,\eta] \cup [1-\eta,1]$,
\vspace*{1mm}
\item $\displaystyle h(s)=\frac{y_i}{2\tau}$ for $s\in [s_i-\tau,s_i+\tau]$, $i=1,\ldots,N$,
\vspace*{1mm}
\item $\displaystyle  |h(s)|\le \big|\frac{y_i}{2\tau}\big|$ for $s\in [s_i-\tau-\eta,s_i-\tau]\cup[s_i+\tau,s_i+\tau+\eta]$, $i=1,\ldots,N$, and
\vspace*{1mm}
\item $|h(s)|\le \eta$ for $s\in [\eta,s_1-\tau-\eta]\cup \big(\bigcup_{i=1}^{N-1} [s_i+\tau+\eta,s_{i+1}-\tau-\eta]\big)\cup [s_N+\tau+\eta,1-\eta]$.
\end{enumerate}
We choose $\tau$ and $\eta$ sufficiently small so that the intervals in {\rm (d)} are nonempty, pairwise disjoint and contained in 
$\mathring I=(0,1)$. Furthermore, if $\tau>0$ is chosen small enough then condition {\rm (b)} ensures that
the following estimate holds for each $i=1,\ldots,N$:
\begin{equation}\label{eq:'2}
    \left| \int_{s_i-\tau}^{s_i+\tau} a(s) \left(h(s) f(s),h(s)^2f(s)^2\right) ds - a(s_i)(y_if(s_i),\frac{y_i^2}{2\tau} f(s_i)^2)\right|<\epsilon'.
\end{equation}
On the other hand, if $\eta>0$ is sufficiently small then  {\rm (a)}, {\rm (c)}, and {\rm (d)} guarantee that
\begin{multline}\label{eq:'3}
    \left|\int_0^{s_1-\tau} a(s)\left(h(s) f(s),h(s)^2f(s)^2\right) ds\right| 
    \\
    + \left|\int_{s_N+\tau}^1 a(s)\left(h(s) f(s),h(s)^2f(s)^2\right) ds\right|
    \\
    + \sum_{i=1}^{N-1}\left|\int_{s_i+\tau}^{s_{i+1}-\tau} a(s)\left(h(s) f(s),h(s)^2f(s)^2\right) ds\right| 
     <\epsilon'.
\end{multline}
Choosing $\epsilon'<\epsilon/(N+2)$, inequalities \eqref{eq:'1}, \eqref{eq:'2}, and \eqref{eq:'3} yield \eqref{eq:approximate1}, 
which concludes the proof.
\end{proof}

With Lemmas \ref{lem:gg2-spray} and \ref{lem:gg2-period} in hand, 
one may easily adapt the arguments in Sections \ref{sec:sprays} and \ref{sec:proof} in order to prove the following proposition.

\begin{proposition}\label{pro:gg2}
Let $M$ be an open Riemann surface, $\qgot\colon H_1(M;\z)\to\c^2$  be a group homomorphism, and $\theta$ be a holomorphic 
$1$-form vanishing nowhere on $M$. Also let $S=K\cup\Gamma\subset M$ be an admissible subset (see Definition \ref{def:admissible}) 
and $u\colon S\to\c_*$ be a function of class $\Ascr(S)$ such that
\[
     \int_\gamma (u,u^2)\theta = \qgot(\gamma)\quad \text{for every closed curve $\gamma\subset S$}.
\] 
Then $u$ may be approximated uniformly on $S$ by nowhere vanishing holomorphic functions $g\colon M\to\c_*$ such that
\[
      \int_\gamma (g,g^2)\theta=\qgot(\gamma)\quad \text{for every closed curve $\gamma\subset M$}.
\]
\end{proposition}

We point out that, when using Lemmas \ref{lem:gg2-spray} and \ref{lem:gg2-period} in order to prove an analogue of Lemma 
\ref{lem:existence-sprays} in the current context, the role of the function $\theta(\gamma_j(s),\dot\gamma_j(s))$ in the proof of that 
lemma is played by the function $a(s)$ in Lemmas \ref{lem:gg2-spray} and \ref{lem:gg2-period}. 
We leave the details of the proof of Proposition \ref{pro:gg2} to the interested reader.

%
%

\begin{proof}[Proof of Theorem \ref{th:nonflat}]
Clearly it suffices to prove the theorem for $n=3$. Let $X\colon M\to\r^3$ be a flat conformal minimal immersion. Without loss of generality 
we may assume that $\di X=(1,\imath,0)\phi_3$ where $\phi_3$ is an exact holomorphic $1$-form vanishing nowhere on $M$.
Choose a nonconstant holomorphic function $g\colon M\to\c_*$ such that $g\phi_3$ and $g^2\phi_3$ are exact $1$-forms 
on $M$; the existence of such $g$ is ensured by Proposition \ref{pro:gg2}. Set
\[
	\Phi_\lambda= \left(1- \lambda^2 g^2,\imath(1+\lambda^2 g^2), 2\lambda g\right) \phi_3, \quad \lambda\in \C.
\]
Note that $\Phi_\lambda$ is an exact holomorphic $1$-form and the map $\Phi_\lambda/\phi_3$ assumes values in the punctured 
null quadric $\Agot_*\subset\c^3$ \eqref{eq:nullquadric} for every $\lambda\in\c$. Thus, fixing a base point $p_0\in M$, 
every $\Phi_\lambda$ provides a conformal minimal immersion $X_\lambda\colon M\to\r^3$ by the formula
\[
      X_\lambda(p)=X(p_0)+2\int_{p_0}^p \Re(\Phi_\lambda),\quad p\in M.
\]
Note that the Gauss map of $X_\lambda$ equals $2\lambda g$ (cf.\ \eqref{eq:EWR}).
Since $\Phi_0=\di X$ and $g$ is nonconstant, we have that $X_0=X$ and $X_1$ is nonflat, and hence the isotopy $X_t\colon M\to\r^n$ 
$(t\in[0,1])$ satisfies the conclusion of the theorem.
\end{proof}

\begin{proof}[Proof of Theorem \ref{th:flat-3}]
Let $\Agot_*\subset \C^3$ be as above (see \eqref{eq:nullquadric}). Fix a  group homomorphism $\pgot \colon H_1(M;\Z) \to \Z_2$.
Choose a continuous map $g\colon M\to \C_*$ such that for every generator $\gamma$ of $H_1(M;\Z)$
we have that $H_1(g)(\gamma)=0\in\Z$ if $\pgot(\gamma)=0\in \Z_2$, and $H_1(g)(\gamma)=1\in\Z$
if $\pgot(\gamma)=1\in \Z_2$. By the Oka principle we can assume that $g$ is holomorphic.
Identifying $\C_*$ with the ray $\C_*\cdotp (1,\igot,0)\subset \Agot_*$, the generator of $H_1(\C_*;\Z)=\Z$ 
maps to the generator of $H_1(\Agot_*;\Z)=\Z_2$, and hence we have that $H_1((1,\imath,0)g)=\pgot \colon H_1(M;\Z) \to \Z_2$.

Pick a nowhere vanishing holomorphic $1$-form $\theta$ on $M$. Lemma \ref{lem:periods} furnishes
a holomorphic function $h\colon M\to\C_*$, homotopic to the constant map $M\to 1$ through maps
$M\to\C_*$, such that $\int_\gamma gh\theta = 0$ holds for every closed curve $\gamma$ in $M$.
Set $\Phi=(1,\imath ,0) hg\theta$; clearly this is an exact holomorphic $1$-form on $M$ with values in $\C^3$, 
the map $f=\Phi/\theta =(1,\imath,0)gh$ assumes values in the ray $\C_*\cdotp (1,\igot,0)\subset \Agot_*$, 
and $H_1(f)=\pgot \colon H_1(M;\Z)\to \Z_2$.
Hence, fixing a point $p_0\in M$, the map $X\colon M\to \R^3$ defined by 
$X(p)= 2\int_{p_0}^p  \Re(\Phi)$ $(p\in M)$ defines a flat conformal minimal immersion 
such that $\di X=\Phi$ and hence $H_1(\di X/\theta)=\pgot$. This completes the proof.
\end{proof}


\subsection*{Acknowledgements}
A.\ Alarc\'on is supported by the Ram\'on y Cajal program of the Spanish Ministry of Economy and Competitiveness.
A.\ Alarc\'on and F.\ J.\ L\'opez are partially supported by the MINECO/FEDER grant no. MTM2014-52368-P, Spain. 
F.\ Forstneri\v c is partially  supported  by the research program P1-0291 and grants J1-5432 and J1-7256 from 
ARRS, Republic of Slovenia. A part of the work on this paper was done while F.\ Forstneri\v c was visiting the 
Department of Geometry and Topology of the  University of Granada, Spain, in March 2016. He wishes to thank this institution for 
the invitation and hospitality.




\vspace*{0.5cm}
\noindent Antonio Alarc\'{o}n

\noindent Departamento de Geometr\'{\i}a y Topolog\'{\i}a e Instituto de Matem\'aticas (IEMath-GR), Universidad de Granada, Campus de Fuentenueva s/n, E--18071 Granada, Spain.

\noindent  e-mail: {\tt alarcon@ugr.es}

\vspace*{0.3cm}
\noindent Franc Forstneri\v c

\noindent Faculty of Mathematics and Physics, University of Ljubljana, Jadranska 19, 1000 Ljubljana, Slovenia; 

\noindent Institute of Mathematics, Physics and Mechanics, Jadranska 19, 1000 Ljubljana, Slovenia.

\noindent e-mail: {\tt franc.forstneric@fmf.uni-lj.si}

\vspace*{0.3cm}
\noindent Francisco J.\ L\'opez

\noindent Departamento de Geometr\'{\i}a y Topolog\'{\i}a e Instituto de Matem\'aticas (IEMath-GR), Universidad de Granada, Campus de Fuentenueva s/n, E--18071 Granada, Spain

\noindent  e-mail: {\tt fjlopez@ugr.es}

\end{document}